\documentclass[11pt]{amsart}
\usepackage{amsfonts}
\usepackage{amssymb}
\usepackage[arrow,matrix]{xy}
\usepackage{amsmath,amssymb, bbm, amscd, amsthm,mathrsfs,hyperref}
\theoremstyle{plain}
\newtheorem{thm}{Theorem}[subsection]
\newtheorem{cor}[thm]{Corollary}
\newtheorem{lem}[thm]{Lemma}
\newtheorem{prop}[thm]{Proposition}
\theoremstyle{definition}
\newtheorem{defn}[thm]{Definition}

\newtheorem{question}[thm]{Question}
\newtheorem{rem}[thm]{Remark}
\newtheorem{thm0}{Theorem}
\newtheorem{question0}[thm0]{Question}

\def\al{\alpha}
\def\bt{\beta}
\def\dt{\delta}

\def\sg{\sigma}
\def\tt{\theta}

\def\lmd{\lambda}

\def\vps{\varepsilon}
\def\om{\omega}

\def\bgm{\Gamma}
\def\bdt{\Delta}

\def\mc{\mathcal}

\def\t{\text}
\def\it{\textit}

\def\ot{\otimes}
\def\op{\oplus}
\def\se{\leqslant}
\def\le{\geqslant}
\def\bxt{\boxtimes}
\def\sq{\square}

\def\ra{\rightarrow}
\def\lra{\longrightarrow}
\def\la{\leftarrow}
\def\xra{\xrightarrow}

\def\Hom{\operatorname {Hom}}
\def\Ext{\operatorname {Ext}}
\def\Tor{\operatorname {Tor}}
\def\Ker{\operatorname {Ker}}

\def\dim{\operatorname {dim}}
\def\id{\operatorname {id}}

\def\Mod{\operatorname {Mod}}

\def\injdim{\operatorname{injdim}}

\def\projdim{\operatorname{projdim}}
\def\gldim{\operatorname{gldim}}
\def\rgldim{\operatorname{rgldim}}
\def\lgldim{\operatorname{lgldim}}
\def\Hdim{\operatorname{Hdim}}
\def\GL{\operatorname {GL}}
\def\SL{\operatorname {SL}}
\def\tr{\operatorname {tr}}
\def\ob{\operatorname {ob}}

\def\kk{\mathbbm{k}}

\def\NN{\mathbb{N}}

\begin{document}
\title[Calabi-Yau property]{\bf  Calabi-Yau property under monoidal Morita-Takeuchi equivalence}

\author{Xingting Wang}
\address{Xingting WANG\newline Department of Mathematics, Temple University, Philadelphia, PA 19122, USA }
\email{xingting@temple.edu}

\author{Xiaolan YU}
\address {Xiaolan YU\newline Department of Mathematics, Hangzhou Normal University, Hangzhou, Zhejiang 310036, China}

\email{xlyu@hznu.edu.cn}

\author{Yinhuo ZHANG}
\address {Yinhuo ZHANG\newline Department WNI, University of Hasselt, Universitaire Campus, 3590 Diepeenbeek,Belgium } \email{yinhuo.zhang@uhasselt.be}

\date{}

\begin{abstract}

Let $H$ and $L$ be two Hopf algebras such that their comodule categories are monoidal equivalent. We prove that if  $H$ is a twisted Calabi-Yau (CY) Hopf algebra, then  $L$ is a twisted CY algebra when it is  homologically smooth. Especially, if $H$ is a Noetherian twisted CY Hopf algebra  and $L$ has finite global dimension, then $L$ is a twisted CY algebra.
\end{abstract}


\maketitle

\section*{Introduction}\label{0}
In noncommutative projective algebraic geometry, the notion of Artin-Schelter (AS) regular algebra $A=\bigoplus_{i\ge 0}A_i$ was introduced in \cite{ars} as a homological analogue of a polynomial algebra. The connected graded noncommutative algebra $A$ is considered as the homogenous coordinate ring of some noncommutative projective space $\mathbb P^n$.

In the lecture note \cite{man}, Manin constructed the quantum general linear group $\mathcal O_A(\GL)$ that universally coacts on an AS regular algebra $A$. Similarly, we can define the quantum special linear group of $A$, denoted by $\mathcal O_A(\SL)$, by requiring the homological codeterminant of the Hopf coaction to be trivial; see \cite[Section 2.1]{ww} for details. As pointed out in \cite{ww}, it is conjectured that these universal quantum groups should possess the same homological properties of $A$, among which the Calabi-Yau (CY) property is the most interesting one since $A$ is always twisted CY according to \cite[Lemma 2.1]{rrz} (see Section 1.2 for the definition of twisted CY algebra). Moreover, many classical quantized coordinate rings can be realized as universal quantum groups associated to AS regular algebras via the above construction \cite{cww, ww}, whose CY property and rigid dualizing complexes have been discussed in \cite{bz,gz}.

Now let us look at a nontrivial example, which is the motivation of our paper. Let $\kk$ be a field. AS regular algebras of global dimension 2 (not necessarily Noetherian) were classified by Zhang in \cite{zh}. They are the algebras (assume they are generated in degree one)
$$A(E)=\kk\langle x_1,x_2,\dots,x_n\rangle/(\sum_{1\se i,j\se n}e_{ij}x_ix_j)$$ for $E=(e_{ij})\in \GL_n(\kk)$ with $n\le2$. It is shown in \cite[Corollary 2.17]{ww} that $\mc{O}_{A(E)}(\SL)\cong\mc{B}(E^{-1})$
as Hopf algebras, where $\mathcal B(E^{-1})$ was defined by Dubois-Violette and Launer \cite{dvl} as the quantum automorphism group of the non-degenerate bilinear form associated to $E^{-1}$. In particular, when
$$
E=
\begin{pmatrix}
0  &  -q\\
1  &  0\\
\end{pmatrix}\ \text{and}\
E^{-1}=E_q=
\begin{pmatrix}
0  &  1\\
-q^{-1}  &  0\\
\end{pmatrix},
\ \text{for some}\ q\in \kk^\times,
$$
we have $A(E)=A_q=k\langle x_1,x_2\rangle/(x_2x_1+qx_1x_2)$ is the quantum plane and $\mathcal{O}_{A_q}(\SL)=\mathcal{B}(E_q)=\mathcal{O}_q(\SL_2)$ is the quantized coordinate ring of $\SL_2(\kk)$.

Two Hopf algebras are called monoidally Morita-Takeuchi equivalent, if their  comodule categories are monoidally equivalent. Bichon obtained that $\mc{B}(E)$ (for any $E\in \GL_n(\kk)$ with $n\ge 2$) and $\mc{O}_q(\SL_2)$ are monoidally Morita-Takeuchi equivalent when  $q^2+\text{tr}(E^tE^{-1})q+1=0$ \cite[Theorem 1.1]{bi3}. By applying this monoidal equivalence, Bichon obtained a free Yetter-Drinfeld module resolution (Definition \ref{defn ydres}) of the trivial Yetter-Drinfeld module $\kk$ over $\mathcal B(E)$ \cite{bi}.
This turns out to be the key ingredient to prove the CY property of $\mathcal B(E)$ \cite{bi,ww}. Note that the quantized coordinate ring $\mathcal O_q(\SL_2)$ is well-known to be twisted CY. Thus it is natural to ask the following question.
\begin{question0}\label{Q1}
Let $H$ and $L$ be two Hopf algebras that are monoidally Morita-Takeuchi equivalent. Suppose $H$ is twisted CY. Is $L$ always twisted CY?
\end{question0}
The monoidal equivalence between the comodule categories of various universal quantum groups have been widely observed in \cite{bi3, bi1,cww, mro} by using the language of cogroupoids. In recent papers \cite{rv1, rv2}, Raedschelders and Van den Bergh proved that, for a Koszul AS regular algebra $A$, the monoidal structure of the comodule category of $\mathcal O_A(\GL)$ only depends on the global dimension of $A$ and not on $A$ itself \cite[Theorem 1.2.6]{rv1}. We expect a positive answer to Question \ref{Q1}, which will play an important role in investigating the CY property of these universal quantum groups associated to AS regular algebras.

The following is our first result, showing that in order to answer Question \ref{Q1}, it suffices to prove that the homologically smooth condition is monoidally Morita-Takeuchi invariant.

\begin{thm0}\label{thm1}(Theorem \ref{thm main'})
Let $H$ and $L$ be two  monoidally Morita-Takeuchi equivalent Hopf algebras. If  $H$ is twisted CY of dimension $d$ and $L$ is homologically smooth, then $L$ is twisted CY of dimension $d$ as well.
\end{thm0}
Note that for Hopf algebras, there are several equivalent descriptions of the  homological smoothness stated in Proposition \ref{prop hom}. Now Question 1 is reduced to the following question.
\begin{question0}\label{Q2}
Let $H$ and $L$ be two monoidally Morita-Takeuchi equivalent Hopf algebras.  Suppose $H$ is homologically smooth. Is $L$ always homologically smooth?
\end{question0}
Though we can not fully answer Question \ref{Q2}, it is true in certain circumstances. The following is the main result in this paper.

\begin{thm0}\label{thm2}(Theorem \ref{thm main1})
Let $H$ be a twisted CY Hopf algebra of dimension $d$, and $L$ a Hopf algebra monoidally Morita-Takeuchi equivalent to $H$.  If one of the following conditions holds,  then $L$ is also twisted CY of dimension $d$.
\begin{enumerate}
\item[(i)] $H$ admits a finitely generated relative projective Yetter-Drinfeld  module resolution for the trivial Yetter-Drinfeld module $\kk$ and $L$ has finite global dimension.
\item[(ii)] $H$ admits a bounded finitely generated relative projective Yetter-Drinfeld  module resolution for the trivial Yetter-Drinfeld module $\kk$.
\item[(iii)] $H$ is Noetherian and $L$ has finite global dimension.
\item[(iv)] $L$ is Noetherian and has finite global dimension.
\end{enumerate}
\end{thm0}

Relative projective Yetter-Drinfeld modules and relative projective Yetter-Drinfeld  module  resolutions will be explained in Section 2.2. The trivial module $\kk$ over $\mathcal O_q(\SL_2)$ admits a finitely generated free Yetter-Drinfeld  resolution of length 3 \cite[Theorem 5.1]{bi}.  Every free Yetter-Drinfeld  module resolution is a relative projective Yetter-Drinfeld module resolution.  According to our result above, this immediately implies that $\mathcal B(E)$ is twisted CY since $\mathcal B(E)$ and $\mathcal O_q(\SL_2)$ are monoidally Morita-Takeuchi equivalent as mentioned above.

Twisted CY algebras, of course, have finite global dimensions. Theorem \ref{thm2} leads to the last question concerning about whether the global dimension is monoidally Morita-Takeuchi invariant. The similar question was asked by Bichon in \cite{bi2} concerning the Hochschild dimension, and the two questions are essentially the same by Proposition \ref{prop dim}.
\begin{question0}\label{Q3}
Let $H$ and $L$ be  two monoidally Morita-Takeuchi equivalent Hopf algebras. Does $\gldim(H)=\gldim(L)$, or at least, $\gldim(H)<\infty$ if and only if $\gldim(L)<\infty$?
\end{question0}
If the answer is positive, then the finite global dimension assumptions in Theorem \ref{thm2} (i), (iii), and (iv)  can be dropped. This will partially answer  Question \ref{Q1} under the assumption that one of the Hopf algebras is Noetherian. As a consequence of our main result, we provide a partial answer under the assumption that both Hopf algebras are twisted CY.

\begin{thm0}\label{thm3}(Corollary \ref{main cor})
Let $H$ and $L$ be  two monoidally Morita-Takeuchi equivalent Hopf algebras. If both $H$ and $L$ are twisted CY, then $\gldim(H)=\gldim(L)$.
\end{thm0}

Monoidal Morita-Takeuchi equivalence can be described by the language of cogroupoids. If $H$ and $L$ are two Hopf algebras such that they are monoidally Morita-Takeuchi equivalent,  then there exists a connected cogroupoid with 2 objects $X,Y$ such that $H=\mc{C}(X,X)$ and $L=\mc{C}(Y,Y)$. In this case, $\mc{C}(X,Y)$ is just the $H$-$L$-biGalois object (see Section 1.1 for details).  Throughout the paper, we will use the language of cogroupoids to discuss Hopf algebras whose comodule categories are monoidally equivalent. We generalize many definitions and results in \cite{bz} to the level of cogroupoids (see Section 2.5). Especially for Hopf-Galois objects, we define the left (resp. right) winding automorphisms of $\mc{C}(X,Y)$ using the homological integrals of $\mc{C}(X,X)$ (resp. $\mc{C}(Y,Y)$). We also generalize the famous Radford $S^4$ formula for finite-dimensional Hopf algebras  to Hopf-Galois object $\mc{C}(X,Y)$ by assuming both $\mc{C}(X,X)$ and $\mc{C}(Y,Y)$ are AS-Gorenstein Hopf algebras (see Theorem \ref{thm inn} and the Remark below). At last, we provide two examples in Section 3. One is the connected cogroupoid associated to $\mc{B}(E)$ and the other is the connected cogroupoid associated to a generic datum of finite Cartan type $(\mc{D},\lambda)$.

\section{Preliminaries}\label{1}

We work over a fixed field $\kk$. Unless stated otherwise all algebras and vector spaces are over $\kk$. The unadorned  tensor $\ot$ means $\ot_\kk$ and $\Hom$ means $\Hom_\kk$.

Given an algebra $A$, we write $A^{op}$ for the
opposite algebra of $A$ and $A^e$ for the enveloping algebra $A\ot
A^{op}$. The category of left (resp. right) $A$-modules is denoted by $\Mod$-$A$ (resp. $\Mod$-$A^{op}$). An $A$-bimodule can be identified with an $A^e$-module, that is, an object in $\Mod$-$A^e$.

For an $A$-bimodule $M$ and two algebra automorphisms $\mu$ and $\nu$, we let $^\mu M^\nu$ denote the $A$-bimodule such that $^\mu M^\nu\cong M$ as vector spaces, and the bimodule structure is given by
$$a\cdot m \cdot b=\mu(a)m\nu(b),$$
for all $a,b\in A$ and $m\in M$. If one of the automorphisms is the identity, we will omit it. It is well-known that   $A^\mu\cong A$ as $A$-bimodules if and only
if $\mu$ is an inner automorphism of $A$.

For a Hopf algebra $H$, as usual, we use the symbols $\bdt$, $\vps$ and $S$ respectively for its comultiplication, counit, and antipode. We use  Sweedler's (sumless) notation for the
comultiplication and coaction of $H$.
The category of right
$H$-comodules is denoted by $\mc{M}^H$. We write ${}_\vps\kk$ (resp. $\kk_\vps$) for the left (resp. right) trivial module defined by the counit $\vps$ of $H$.

\subsection{Cogroupoid}
We first recall the definition of a cogroupoid.
\begin{defn}A \it{cocategory} $\mc{C}$ consists of:
\begin{itemize}
\item A set of objects $\ob(\mc{C})$.
\item For any $X,Y\in \ob(\mc{C})$, an algebra $\mc{C}(X,Y)$.
\item For any $X,Y,Z\in \ob(\mc{C})$, algebra homomorphisms
$$\bdt^Z_{XY}:\mc{C}(X,Y)\ra \mc{C}(X,Z)\ot \mc{C}(Z,Y) \t{ and } \vps_X:\mc{C}(X,X)\ra \kk$$
such that for any $X,Y,Z,T\in \ob(\mc{C})$, the following diagrams commute:
$$\begin{CD}\mc{C}(X,Y)@>\bdt^Z_{X,Y}>>\mc{C}(X,Z)\ot \mc{C}(Z,Y)\\
@V\bdt^T_{X,Y}VV@V\bdt^T_{X,Z}\ot 1VV\\
\mc{C}(X,T)\ot \mc{C}(T,Y)@>1\ot \bdt^Z_{T,Y}>>\mc{C}(X,T)\ot \mc{C}(T,Z)\ot \mc{C}(Z,Y)\end{CD}$$
\xymatrix
{\mc{C}(X,Y)\ar@{=}[rd]\ar^{\bdt^Y_{X,Y}}[d]&\\
\mc{C}(X,Y)\ot\mc{C}(Y,Y)\ar^-{1\ot \vps_Y}[r]&\mc{C}(X,Y)}\hspace{2mm}
\xymatrix
{\mc{C}(X,Y)\ar@{=}[rd]\ar^{\bdt^X_{X,Y}}[d]&\\
\mc{C}(X,X)\ot\mc{C}(X,Y)\ar^-{\vps_X \ot 1 }[r]&\mc{C}(X,Y).}
\end{itemize}
\end{defn}
Thus a cocategory with one object is just a bialgebra.

A cocategory $\mc{C}$ is said to be \it{connected} if $\mc{C}(X,Y)$ is a nonzero algebra for any $X,Y\in \ob(\mc{C})$.

\begin{defn}\label{defn cogoupoid}A \it{cogroupoid} $\mc{C}$ consists of a cocategory $\mc{C}$ together with, for any $X,Y\in \ob(\mc{C})$, linear maps
$$S_{X,Y}:\mc{C}(X,Y)\longrightarrow \mc{C}(Y,X)$$
such that for any $X,Y\in \mc{C}$, the following diagrams commute:
$$\xymatrix{\mc{C}(X,X)\ar[d]_{\bdt_{X,X}^Y}\ar[r]^-{\vps_X}&\kk\ar[r]^-u&\mc{C}(X,Y)\\
\mc{C}(X,Y)\ot\mc{C}(Y,X)\ar[rr]^{1\ot S_{Y,X}}&&\mc{C}(X,Y)\ot\mc{C}(X,Y)\ar[u]^m}$$
$$\xymatrix{\mc{C}(X,X)\ar[d]_{\bdt_{X,X}^Y}\ar[r]^-{\vps_X}&\kk\ar[r]^-u&\mc{C}(Y,X)\\
\mc{C}(X,Y)\ot\mc{C}(Y,X)\ar[rr]^{S_{X,Y}\ot 1}&&\mc{C}(Y,X)\ot\mc{C}(Y,X)\ar[u]^m}$$
\end{defn}
From the definition, we can see that $\mc{C}(X,X)$ is a Hopf algebra for each object $X\in \mc{C}$.

We use Sweedler's notation for cogroupoids. Let $\mc{C}$ be a cogroupoid. For any $a^{X,Y}\in \mc{C}(X,Y)$, we write
$$\bdt^Z_{X,Y}(a^{X,Y})=a^{X,Z}_1\ot a^{Z,Y}_2.$$
Now the cogroupoid axioms are
$$(\bdt_{X,Z}^T\ot 1)\circ \bdt^Z_{X,Y}(a^{X,Y})=a^{X,T}_1\ot a^{T,Z}_2\ot a^{Z,Y}_3=(1\ot \bdt_{T,Y}^Z)\circ \bdt^T_{X,Y}(a^{X,Y});$$
$$\vps_X(a_1^{X,X})a_2^{X,Y}=a^{X,Y}=a_1^{X,Y}\vps_Y(a_2^{Y,Y});$$
$$S_{X,Y}(a_1^{X,Y})a_2^{Y,X}=\vps_X(a_1^{X,X})1=a_1^{X,Y}S_{Y,X}(a_2^{Y,X}).$$

The following is Proposition 2.13 in \cite{bi1}. It describes properties of the ``antipodes''.
\begin{lem}
Let $\mc{C}$ be a cogroupoid and let $X,Y\in\ob(\mc{C})$.
\begin{enumerate}
\item[(i)] $S_{Y,X}:\mc{C}(Y,X)\ra \mc{C}(X,Y)^{op}$ is an algebra homomorphism.
\item[(ii)] For any $Z\in \ob(\mc{C})$ and $a^{Y,X}\in\mc{C}(Y,X)$,
$$\bdt_{X,Y}^Z(S_{Y,X}(a^{Y,X}))=S_{Z,X}(a_2^{Z,X})\ot S_{Y,Z}(a_1^{Y,Z}).$$

\end{enumerate}
\end{lem}

For other basic properties of cogroupoids, we refer to \cite{bi1}.

In \cite{bi1}, Bichon reformulated Schauenburg's results in \cite{sch} by cogroupoids. This theorem shows that to discuss two Hopf algebras with monoidally equivalent comodule categories is equivalent to discuss connected cogroupoids. In what follows, without otherwise stated, we assume that the cogroupoids mentioned are \textit{connected}.
\begin{thm}\cite[Theorem 2.10, 2.12]{bi1}\label{thm ga}
Let $\mc{C}$ be a connected cogroupoid. Then for any $X,Y\in\mc{C}$, we have equivalences of monoidal categories that are inverse of each other
$$\begin{array}{rclcrcl}\mc{M}^{\mc{C}(X,X)}&\cong^\ot& \mc{M}^{\mc{C}(Y,Y)}&&\mc{M}^{\mc{C}(Y,Y)}&\cong^\ot& \mc{M}^{\mc{C}(X,X)}\\
V&\longmapsto&V\sq_{\mc{C}(X,X)}\mc{C}(X,Y)&&V&\longmapsto&V\sq_{\mc{C}(Y,Y)}\mc{C}(Y,X)
\end{array}$$
Conversely, if $H$ and $L$ are Hopf algebras such that $\mc{M}^H\cong ^\ot \mc{M}^L$, then there exists a connected cogroupoid with 2 objects $X,Y$ such that $H=\mc{C}(X,X)$ and $L=\mc{C}(Y,Y)$.
\end{thm}

This monoidal equivalence can be extended to categories of Yetter-Drinfeld modules.
\begin{lem}\cite[Proposition 6.2]{bi1}\label{lem yd}
Let $\mc{C}$ be a cogroupoid, $X,Y\in\ob(\mc{C})$ and $V$ a right $\mc{C}(X,X)$-module.
\begin{enumerate}
\item[(i)] $V\ot \mc{C}(X,Y)$ has a right $\mc{C}(Y,Y)$-module structure defined by
$$ (v\ot a^{X,Y})\la b^{Y,Y} =  v \cdot b_2^{X,X}\ot S_{Y,X}(b_1^{Y,X}) a^{X,Y} b_3^{X,Y}.$$
Together with the right $\mc{C}(Y,Y)$-comodule structure defined by $1\ot \bdt_{X,Y}^Y$,  $V\ot \mc{C}(X,Y)$ is a Yetter-Drinfeld module over $\mc{C}(Y,Y)$.
\item[(ii)] If moreover $V$ is a Yetter-Drinfeld module, then $V\square_{\mc{C}(X,X)}\mc{C}(X,Y)$ is a  Yetter-Drinfeld submodule of $V\ot \mc{C}(X,Y)$.
\end{enumerate}
\end{lem}

\begin{thm}\cite[Theorem 6.3]{bi1}
Let $\mc{C}$ be a connected cogroupoid. Then for any $X,Y\in \ob(\mc{C})$, the functor
$$\begin{array}{rcl}\mc{YD}_{\mc{C}(X,X)}  ^{\mc{C}(X,X)}&\lra&\mc{YD}_{\mc{C}(Y,Y)} ^{\mc{C}(Y,Y)}\\
V&\longmapsto&V\sq_{\mc{C}(X,X)}\mc{C}(X,Y)
\end{array}$$
is a monoidal equivalence.
\end{thm}

\subsection{Calabi-Yau algebras}
In this subsection, we recall the definition of (twisted) Calabi-Yau algebras.
\begin{defn}\label{defn tcy}An algebra $A$ is called  a \it{twisted Calabi-Yau algebra  of dimension
$d$} if
\begin{enumerate}\item[(i)] $A$ is \it{homologically smooth}, that is, $A$ has
a bounded resolution by finitely generated projective
$A^e$-modules; \item[(ii)] There is an automorphism $\mu$ of $A$ such that
\begin{equation}\Ext_{A^e}^i(A,A^e)\cong\begin{cases}0,& i\neq d
\\A^\mu,&i=d\end{cases}\end{equation}
as $A^e$-modules.
\end{enumerate}
If such an automorphism $\mu$ exists, it is unique up to an inner automorphism and is called the \it{Nakayama automorphism} of $A$. In the definition, the dimension $d$ is usually called the Calabi-Yau dimension of $A$.  A \it{Calabi-Yau algebra} in the sense of Ginzburg \cite{g2} is a twisted Calabi-Yau algebra whose Nakayama automorphism is an inner automorphism. In what follows, Calabi-Yau is abbreviated to CY for short.
\end{defn}

Twisted CY algebras include CY algebras as a subclass. They are the natural algebraic analogues of the Bieri-Eckmann duality groups \cite{be}.  The twisted CY property of noncommutative algebras has been studied under other names for many years, even before the definition of a CY algebra. Rigid dualizing complexes of noncommutative algebras were studied in \cite{vdb}. The twisted CY property
was called ``rigid Gorenstein'' in \cite{bz} and was called ``skew Calabi-Yau'' in a recent paper \cite{rrz}.

\section{Calabi-Yau Property}\label{2}

\subsection{Hopf algebra preparations}In this subsection, we provide two results about Hopf algebras as preparations. They may be well-known, but we could not find any reference, so we give  a complete account of proofs here. We donot require bijectiveness of antipode or Noetherianness of a Hopf algebra.

First we want to show that for a Hopf algebra,  the left global dimension always equals the right global dimension.

Let $H$ be a Hopf algebra.  We denote the left global dimension, the right global dimension and the Hochschild dimension of $H$ by $\lgldim(H)$, $\rgldim(H)$ and $\Hdim(H)$, respectively. The left adjoint functor $L:\Mod$-$H^e\ra \Mod$-$H$ is defined by the algebra homomorphism $(\id\ot S)\circ \bdt:H\ra H^e$. Similarly, the algebra homomorphism $\tau\circ(S\ot \id)\circ \bdt:H\ra (H^e)^{op}=H^e$ defines the right adjoint functor  $R:\Mod$-$(H^e)^{op}\ra \Mod$-$H^{op}$, where $\tau: H^{op}\ot H\ra H\ot H^{op}$ is the flip map.  Let $M$ be an $H$-bimodule. Then $L(M)$ is a left $H$-module defined by the action
$$x\ra m =x_1mS(x_2),$$
for any $x\in H$.
While $R(M)$ is a right $H$-module defined by  the action
$$ m \la x=S(x_1)m x_2,$$
for any $x\in H$.

The algebra $H^e$ is a left and right $H^e$-module respectively as in the following ways:
\begin{equation}\label{HHlm}(a\ot b)\ra (x\ot y)=a  x\ot y  b,\end{equation}
and \begin{equation}\label{HHrm}(x\ot y)\la(a\ot b) =x  a\ot b  y.\end{equation}
for any $x\ot y$ and $a\ot b \in H^e$. So $L(H^e)$ and $R(H^e)$ are $H$-$H^e$ and $H^e$-$H$-bimodules, where the corresponding $H$-module structures are given by
$$a\ra (x\ot y)=a_1  x\ot y  S(a_2)$$
and
$$(x\ot y) \leftarrow a=xa_2\ot S(a_1)y  $$
for any $a\in H$ and $x\ot y\in H^e$, respectively.

Let ${}_*H\ot H$ be the free left $H$-module, where the structure is given by the left multiplication to the first factor $H$. Similarly, let
$H_*\ot H$ be the free right $H$-module defined by the right multiplication to the first factor $H$. Moreover, we give ${}_*H\ot H$ a right $H^e$-module structure such that
\begin{equation}\label{HHEr} (x\ot y)\leftarrow (a\otimes b)=xa_1\ot by  S^2(a_2)\end{equation}
and $H_*\ot H$ a left $H^e$-module structure via
\begin{equation}\label{HHEl} (a\otimes b)\ra (x\ot y)=a_2x\ot S^2(a_1)yb\end{equation}
for any $x\otimes y\in {}_*H\ot H$ or $H_*\ot H$ and $a\otimes b\in H^e$.

%

\begin{lem}\label{lem free}
Retain the above notations. Then we have
\begin{itemize}
\item[(i)] $L(H^e)\cong {}_*H\otimes H$ as $H$-$H^e$-bimodules.
\item[(ii)] $R(H^e)\cong H_*\otimes H$ as $H^e$-$H$-bimodules.
\end{itemize}
\end{lem}
\begin{proof}
It is straightforward to check the corresponding isomorphisms of bimodules are given by the following four homomorphisms.
$$L(H^e) \ra {}_*H\ot H, \;\;x\ot y \mapsto x_1 \ot y S^2(x_2)$$
with inverse
$${}_*H\ot H \ra L(H^e), \;\;x\ot y \mapsto   x_1 \ot y S(x_2),$$
and
$$ R(H^e) \ra H_*\ot H,\;\;x\ot y  \mapsto  x_2 \ot S^2(x_1)y $$
with inverse
$$
 H_*\ot H \ra R(H^e),\;\;x\ot y \mapsto  x_2 \ot S(x_1) y .$$
\end{proof}

The following is Lemma 2.4 in \cite{bz}. For the sake of completeness, we include a proof here.
\begin{lem}\label{lem rlk} Let $H$ be a Hopf algebra and $M$ an $H$-bimodule.
\begin{enumerate}
\item[(i)] $\Ext^i_{H^e}(H,M)\cong \Ext^i_H({}_\vps\kk,L(M)),$
for all $i\le 0$.
\item[(ii)] $\Ext^i_{H^e}(H,M)\cong \Ext^i_{H^{op}}(\kk_\vps,R(M)),$
for all $i\le 0$.
\end{enumerate}
\end{lem}
\proof We only prove (i), the proof of (ii) is quite similar. We view $H^e$ as an $H^e$-$H$-bimodule, where the left $H^e$-action is given by \eqref{HHlm} and the right $H$-action is given by
$$(x\ot y)\la a=xa_1\ot S(a_2)y$$
for any $x\ot y\in H^e$ and $a\in H$. We simply write this bimodule as ${}_{H^e}H^e_H$. Note that the bimodule ${}_{H^e}H^e_H$ is a free right $H$-module, where the isomorphisms are given by
$${}_{H^e}H^e_H \ra H_*\ot H, \;\;x\ot y \mapsto x_1 \ot x_2 y,$$
with inverse
$$ H_*\ot H\ra {}_{H^e}H^e_H, \;\;x\ot y \mapsto x_1 \ot S(x_2) y.$$
It is easy to check that the functor $L:\Mod$-$H^e\ra \Mod$-$H$ is isomorphic to the functor $\Hom_{H^e}({}_{H^e}H^e_H,-)$. Hence, the functor $L$ is exact. Moreover, we have $$\begin{array}{rcl}\Hom_H(-,L(M))&\cong&\Hom_H(-,\Hom_{H^e}({}_{H^e}H^e_H,M))\\
&\cong &\Hom_{H^e}({}_{H^e}H^e_H\ot_H-,M)
\end{array}$$
for any $H$-bimodule $M$. So if $M$ is an injective bimodule, $L(M)$ is an injective left $H$-module. That is, the functor $L$ preserves injectivity.  Since $\Hom_{H^e}(H, M)\cong \Hom_H({}_\vps\kk, L(M))$, we have $\Ext^i_{H^e}(H,M)\cong \Ext^i_H({}_\vps\kk,L(M)),$ for any $i\le0$.\qed

It is well-known that there is an equivalence of categories between the category of left $H^e$-modules and the category of right $H^e$-modules for $(H^e)^{op}=H^e$. As a consequence, $\Ext_{H^e}^i(H,H^e)$ can be computed both by using the left and the right $H^e$-module structures on $H^e$ defined in \eqref{HHlm} and \eqref{HHrm}.

\begin{prop}\label{prop lrAS}
For any Hopf algebra $H$, we have
$$\Ext_{H^e}^i(H,H^e)\cong \Ext_H^i({}_\vps\kk,H)\otimes H\cong \Ext^i_{H^{op}}(\kk_\vps,H)\otimes H$$
as $H^e$-modules for all $i\ge 0$, where the $H^e$-module structures on $\Ext_H^i({}_\vps\kk,H)\otimes H$ and on $\Ext^i_{H^{op}}(\kk_\vps,H)\otimes H$ are induced by \eqref{HHEr} and \eqref{HHEl} respectively.
\end{prop}
\begin{proof}
We apply Lemmas \ref{lem free} and \ref{lem rlk} to get
$$\Ext_{H^e}^i(H,H^e)\cong \Ext_H^i({}_\vps\kk,L(H^e))\cong \Ext_H^i({}_\vps\kk, {}_*H\otimes H)\cong \Ext_H^i({}_\vps\kk,H)\otimes H$$
and
$$\Ext_{H^e}^i(H,H^e)\cong \Ext_H^i(\kk_\vps,R(H^e))\cong \Ext_H^i({}_\vps\kk, H_*\otimes H)\cong \Ext_H^i(\kk_\vps,H)\otimes H.$$
Moreover, all the homomorphisms above are isomorphisms of $H^e$-modules.
\end{proof}

\begin{prop}\label{prop dim}
Let $H$ be a Hopf algebra. Then $$\projdim \kk_\vps=\projdim {}_\vps\kk=\rgldim(H)=\lgldim(H)=\Hdim(H).$$
\end{prop}
\proof That $\projdim \kk_\vps=\rgldim(H)$ and $\projdim {}_\vps\kk=\lgldim(H)$ follows from \cite[Section 2.4]{ll}. We know from \cite[IX.7.6]{ce} that $\rgldim(H)$ and $\lgldim(H)$ are bounded by $\Hdim(H)$. Let $M$ be any $H$-bimodule. The isomorphism $\Ext^i_{H^e}(H,M)\cong \Ext^i_H({}_\vps\kk,L(M))$, $i\le 0$ shows that $\Hdim(H)\se \lgldim (H)$. Similarly, the isomorphism $\Ext^i_{H^e}(H,M)\cong \Ext^i_H(\kk_\vps,R(M))$, $i\le 0$ shows that $\Hdim(H)\se \rgldim(H)$. So we obtain that $\rgldim(H)=\lgldim(H)=\Hdim(H)$. In conclusion, we obtain that
$$\projdim \kk_\vps=\projdim {}_\vps\kk=\rgldim(H)=\lgldim(H)=\Hdim(H).$$\qed

Therefore, for any Hopf algebra $H$, there is no need to distinguish its left global dimension and right global dimension. In the following, we denote the global dimension of $H$ by $\gldim(H)$.

Next we want to show that to see whether  a Hopf algebra $H$ is homologically smooth it is enough to investigate the projective resolution of  the trivial module.

\begin{prop}\label{prop hom}
Let $H$ be a Hopf algebra. The following assertions are equivalent:
\begin{enumerate}
\item[(i)] The algebra $H$ is homologically smooth.
\item[(ii)] The left trivial module ${}_\vps\kk$ admits a  bounded projective resolution with each term finitely generated.
\item[(iii)] The right trivial module $\kk_\vps$ admits a  bounded projective resolution with each term finitely generated.
\end{enumerate}
\end{prop}
\proof We only need to show that (i) and (ii) are equivalent. (i)$\Leftrightarrow$(iii) can be proved symmetrically. The proof is quite similar to the proof of Lemma 5.2 in \cite{bz}.

(i)$\Rightarrow$(ii) This  is obvious. Suppose that  $H$  is homologically smooth. That is, $H$ has a resolution
$$0\ra P_n\ra P_{n-1}\ra \cdots \ra P_1\ra P_0\ra H\ra 0$$
such that  each term is a  finitely generated projective $H^e$-module. Then
$$0\ra P_n\ot_H {}_\vps\kk\ra P_{n-1}\ot_H {}_\vps\kk\ra \cdots \ra P_1\ot_H {}_\vps\kk\ra P_0\ot_H {}_\vps\kk\ra  {}_\vps\kk\ra 0$$
is  a  bounded projective resolution  of ${}_\vps\kk$ with each term finitely generated as left $H$-module.

(ii)$\Rightarrow$(i) By the same proof of Lemma 5.2 (c) in \cite{bz}, we can show that for any $H$-bimodule $M$, if $\projdim_{H^e}M<\infty$ and $\Ext^i_{H^e}(M,-)$ commutes with inductive direct limits for all $i$, then $M$ has a bounded resolution of finitely generated projective $H^e$-modules. So to prove that $H$ is homologically smooth,  we only need to show that $H$ has finite Hochschild dimension and $\Ext^i_{H^e}(H,-)$ commutes with inductive  direct limits for all $i$.

Since ${}_\vps\kk$ admits a  bounded projective resolution with each term finitely generated, the projective dimension of ${}_\vps\kk$ is finite.  Proposition \ref{prop dim} shows that the Hochschild dimension equals projective dimension of ${}_\vps\kk$. Therefore, $H$ has finite Hochschild dimension.

Let $M$ be an $H$-bimodule. By Lemma \ref{lem rlk},
$\Ext^i_{H^e}(H,M)=\Ext_H^i({}_\vps\kk,L(M))$ for all $i$. The trivial module ${}_\vps\kk$ admits a  bounded projective resolution with each term finitely generated. Hence, $\Ext_H^i({}_\vps\kk,-)$ commutes with inductive  direct limits. It is mentioned in the proof of Lemma \ref{lem rlk} that the functor $L:\Mod$-$H^e\ra \Mod$-$H$ is isomorphic to the functor $\Hom_{H^e}({}_{H^e}H^e_H,-)$. We can view $H$ as a subalgebra of $H^e$ via the right $H$-module structure since it is free, that is $a\mapsto a_1\otimes S(a_2)$ for any $a\in H$. Hence $L(-)$ is just a restriction functor, and it commutes with inductive direct limits as well. Therefore, $\Ext^i_{H^e}(H,-)$ commutes with inductive direct limits for all $i$. Now we conclude  that  $H$ is homologically smooth. \qed

\subsection{Artin-Schelter Gorenstein Hopf algebras} We first recall the definition of an Artin-Schelter (AS) Gorenstein algebra.
\begin{defn}(cf. \cite[defn. 1.2]{bz})\label{defn as} Let $H$ be a Hopf algebra.
\begin{enumerate}
\item[(i)] The Hopf algebra $H$ is said to be \it{left
AS-Gorenstein}, if
\begin{enumerate}
\item[(a)] $\injdim  {_HH}=d<\infty$,
\item[(b)] $\Ext_H^i({_\vps\kk},{H})=0$ for $i\neq d$ and $\Ext_H^d({_\vps\kk},{H})=\kk$.
\end{enumerate}
\item[(ii)] The Hopf algebra $H$ is said to be \it{right
AS-Gorenstein}, if
\begin{enumerate}
\item[(c)] $\injdim  {H_H}=d<\infty$,
\item[(d)] $\Ext_{H^{op}}^i({\kk_\vps},{H})=0$ for $i\neq d$ and $\Ext_{H^{op}}^d({\kk_\vps},{H})=\kk$.
\end{enumerate}
\item[(iii)] If $H$ is both left and right AS-Gorenstein (relative to the same augmentation map $\vps$), then $H$ is called \textit{AS-Gorenstein}.
\item[(iv)] If, in addition, the global dimension of $H$ is finite, then $H$ is called  \it{AS-regular}.
\end{enumerate}
\end{defn}

\begin{rem}
In above definitions, we do not require the Hopf algebra $H$ to be Noetherian. For AS-regularity,  the right global dimension always equals the left global dimension by Proposition \ref{prop dim}. Moreover, when $H$ is AS-Gorenstein, the right injective dimension always equals the left injective dimension, which are both given by the integer $d$ such that $\Ext_{H^e}^d(H,H^e)\neq 0$ by Proposition \ref{prop lrAS}.
\end{rem}

Homological integrals for an AS-Gorenstein Hopf algebra introduced in \cite{lwz} is a generalization of integrals for finite dimensional Hopf algebras \cite{sw}. The concept was further extended to any AS-Gorenstein algebra in \cite{bz}.

Let $A$ be a left AS-Gorenstein algebra of injective dimension $d$. One sees that $\Ext^d_A({}_\vps\kk, A)$ is a one dimensional right $A$-module. Any nonzero element in $\Ext^d_A({}_\vps\kk, A)$ is called a \emph{left  homological integral} of $A$. Usually, $\Ext^d_A({}_\vps\kk, A)$ is denoted  by $\int^l_A$. Similarly, if $A$ is a right AS-Gorenstein algebra of injective dimenson $d$, any nonzero element in $\Ext^d_{A^{op}}(\kk_\vps, A )$ is called a \emph{right homological integral}. And $\Ext^d_{A^{op}}(\kk_\vps, A )$ is denoted by $\int^r_A$. Abusing the language slightly, $\int^l_A$ (resp. $\int^r_A$) is also called the left (resp. right) homological integral.

A Noetherian Hopf algebra $H$ is AS-regular in the sense of \cite[Definition 1.2]{bz} if and only if $H$ is twisted CY (\cite[Lemma 1.3]{rrz}). If $H$ is not necessarily Noetherian, we have the following result.

\begin{prop}\label{prop ascy}
Let $H$ be a Hopf algebra with bijective antipode such that it is homologically smooth. Then the followings are equivalent.
\begin{enumerate}
\item[(i)] $H$ is a twisted CY algebra of dimension $d$.
\item[(ii)] There is an integer $d$ such that $\Ext_{H}^i({}_\vps\kk,H)=0$ for $i\neq d$ and $\dim\Ext_{H}^d({}_\vps\kk,H)=1$.
\item[(iii)] There is an integer $d$ such that $\Ext_{H^{op}}^i(\kk_\vps,H)=0$ for  $i\neq d$ and $dim\Ext_{H^{op}}^d(\kk_\vps,H)=1.$
\item[(iv)] $\Ext_{H}^i({}_\vps\kk,H)$ and $\Ext_{H^{op}}^i(\kk_\vps,H)$ are finite dimensional for $i\le 0$ and there is an integer $d$ such that $\dim \Ext^i_H({}_\vps\kk,H)=\dim \Ext^i_{H^{op}}(\kk_\vps,H)=0$ for $i>d$, and $\dim \Ext^d_H({}_\vps\kk,H)\neq 0$ or $\dim \Ext^d_{H^{op}}({}\kk_\vps,H)\neq 0$.
\end{enumerate}
In these cases, we have $\gldim(H)=\injdim  {H_H}=\injdim  {H_H}=d$.
\end{prop}
\proof (i)$\Rightarrow$(ii),(iii) This proof can be found for example in \cite[Lemma 2.15]{yvz}.

(ii)$\Rightarrow$ (i) By Proposition \ref{prop lrAS}, $\Ext_{H^e}^i(H,H^e)\cong\Ext_H^i({}_\vps\kk,H)\ot H$ for all $i\ge 1$ as $H^e$-modules.
Since $\Ext_H^d({}_\vps\kk, H)$ is  a one dimensional right $H$-module, we simply write it as $\kk_\xi$, for some algebra homomorphism $\xi:H\ra\kk$.  Therefore, $\Ext_{H^e}^i(H,H^e)=0$ for $i\neq d$ and $\Ext_{H^e}^d(H,H^e)\cong \kk_\xi\ot H\overset{(a)}\cong H^\mu$, where $\mu$ is defined by $\mu(h)=\xi(h_1)S^{2}(h_2)$ for any $h\in H$. The isomorphism (a) holds because the $H^e$-module structure on $\kk_\xi\ot H$ is induced by the equation (\ref{HHEr}) according to Proposition \ref{prop lrAS}. Moreover, it is easy to check that $\mu$ is an algebra automorphism of $H$ with inverse given by $\mu^{-1}(h)=\xi(S(h_1))S^{-2}(h_2)$ for any $h\in H$.

(iii)$\Rightarrow$(i) The proof is similar to that of (ii)$\Rightarrow$ (i).

(ii), (iii)$\Rightarrow$(iv)  This is obvious.

(iv)$\Rightarrow$(ii), (iii) The proof of \cite[Lemma 3.2]{bz} works generally for this case. Suppose $\dim \Ext^d_H({}_\vps\kk,H)\neq 0$, and it is similar for $\dim \Ext^d_{H^{op}}({}\kk_\vps,H)\neq 0$. Since $H$ is homologically smooth, by Proposition \ref{prop hom} and \cite[Lemma 1.11]{brg}, we can apply the Ischebeck's spectral sequence
\[
\Ext^p_{H^{op}}(\Ext_H^{-q}({}_\vps\kk,H),H)\Longrightarrow \Tor^H_{-p-q}(H,{}_\vps\kk).
\]
to obtain $\dim \Ext^i_{H^{op}}(\kk_\vps,H)=0$ for $i\neq d$. From the proof of \cite[Lemma 1.11]{brg}, $\dim\Ext_H^d(M,H)=\dim M\cdot \dim\Ext_H^d({}_\vps\kk,H)$  for any finite dimensional left $H$-module $M$. Thus by the finite dimensional assumption,
$$\dim\Ext_H^d(\Ext_{H^{op}}^d(\kk_\vps,H),H)=\dim \Ext_{H^{op}}^d(\kk_\vps,H)\cdot \dim\Ext_H^d({}_\vps\kk,H).$$
Again by the Ischebeck's spectral sequence,
$\Ext_H^d(\Ext_{H^{op}}^d(\kk_\vps,H),H)\cong \kk$. Hence,
$$\dim \Ext^d_H({}_\vps\kk,H)=\dim \Ext^d_{H^{op}}(\kk_\vps,H)=1.$$
Now (ii) and (iii) are proved.

Finally, we can apply the same proof of \cite[Proposition 2.2]{bt} to show that for a twisted CY Hopf algebra $H$ of dimension $d$, we have $\Hdim(H)=d$. Hence $\gldim(H)=d$ by Proposition \ref{prop dim}. The equality of the injective dimension of $H$ is easy to see since it is always bounded by $\gldim(H)=d$ and we have $\dim \Ext^d_H({}_\vps\kk,H)\neq 0$ or $\dim \Ext^d_{H^{op}}({}\kk_\vps,H)\neq 0$.

\qed

\begin{cor}
Let $H$ be a Hopf algebra with bijective antipode. Then the following are equivalent
\begin{itemize}
\item [(i)] $H$ is twisted CY.
\item [(ii)] $H$ is left AS-Gorenstein and the left trivial module ${}_\vps\kk$ admits a bounded projective resolution with each term finitely generated.
\item [(iii)] $H$ is right AS-Gorenstein and the right trivial module $\kk_\vps$ admits a bounded projective resolution with each term finitely generated.
\end{itemize}
\end{cor}
\begin{proof}
It follows from Proposition \ref{prop hom} and Proposition \ref{prop ascy}.
\end{proof}

\subsection{Yetter-Drinfeld modules} In this subsection, we recall some definitions related to Yetter-Drinfeld modules.
\begin{defn}
Let $H$ be a Hopf algebra. A
\textit{(right-right) Yetter-Drinfeld module} $V$ over $H$ is simultaneously a
right $H$-module and a right $H$-comodule satisfying the compatibility
condition $$\dt( v\cdot h) = v_{(0)}\cdot h_2
\ot S(h_1) v_{(1)}h_3,
$$ for any $v\in V$, $h\in H$.
\end{defn}

We denote by $\mc{YD}^H_H$ the category of Yetter-Drinfeld modules
over $H$ with morphisms given by $H$-linear and $H$-colinear maps.  Endowed with the usual tensor product of modules and comodules, it
is a monoidal category, with unit the trivial Yetter-Drinfeld module $\kk$.

We can always construct a Yetter-Drinfeld module from a right comodule.

\begin{lem}\cite[Proposion 3.1]{bi} Let $H$ be a Hopf algebra and $V$ a right $H$-comodule.  Endow $V\ot H$ with the right $H$-module structure defined
 by multiplication on the right. Then the linear map
$$\begin{array}{rcl}V\ot H&\ra& V\ot H\ot H\\
v\ot h&\mapsto &v_{(0)}\ot h_2\ot S(h_1)v_{(1)} h_3\end{array}$$
endows $V\ot H$ with a right $H$-comodule structure, and with a right-right Yetter-Drinfeld module structure.
We denote by $V\boxtimes H$ the resulting Yetter-Drinfeld module.
\end{lem}

\begin{defn}\cite[Definition 3.5]{bi}\label{defn fyd}
Let $H$ be a Hopf algebra. A  Yetter-Drinfeld module over $H$ is said to be \textit{free} if it is isomorphic to $V\bxt H$ for some right $H$-comodule $V$.
\end{defn}

A free Yetter-Drinfeld module is obviously free as a right $H$-module. We call a free Yetter-Drinfeld module $V\bxt H$ \it{finitely generated} if $V$ is finite dimensional.

In \cite{bi2}, Bichon introduced the notion
of relative projective Yetter-Drinfeld module, corresponding to  the notion of
relative projective Hopf bimodule considered in \cite{ss} via the monoidal equivalence between Yetter-Drinfeld modules and Hopf bimodules.

\begin{defn}\cite[Definition 4.1]{bi2}\label{defn rpyd}
Let $H$ be a Hopf algebra. A  Yetter-Drinfeld module $P$ over $H$ is said to be \textit{relative projective} if the functor $\Hom_{\mc{YD}^H_H}(P,-)$ transforms exact sequences of Yetter-Drinfeld modules that splits as sequences of comodules to exact sequences of vector spaces.
\end{defn}

The following lemma shows that  relative projective  Yetter-Drinfeld modules are precisely  direct summands of free Yetter-Drinfeld modules.

\begin{lem}\cite[Proposition 4.2]{bi2}\label{lem reproj}
Let $P$ be a Yetter-Drinfeld module over a Hopf algebra $H$. The following assertions are equivalent.
\begin{enumerate}
\item $P$ is relative projective.
\item Any epimorphism of Yetter-Drinfeld modules $f: M\ra P$ that admits a comodule section  admits a Yetter-Drinfeld module section.
\item $P$ is a direct summand of a free Yetter-Drinfeld module.
\end{enumerate}
\end{lem}

It is clear that a relative projective Yetter-Drinfeld module is a projective module. We call a relative projective Yetter-Drinfeld module \textit{finitely generated} if it is a direct summand of a finitely generated free Yetter-Drinfeld module.

\begin{defn}\label{defn ydres}
Let $H$ be a Hopf algebra and let $M\in \mc{YD}^H_H$. A \textit{free (resp. relative projective) Yetter-Drinfeld module resolution} of $M$ consists of a complex of free (resp. relative projective) Yetter-Drinfeld modules
$$\textbf{P}_*:\cdots\ra P_{i+1}\ra P_i\ra \cdots\ra P_1\ra P_0\ra 0$$
for which there exists a Yetter-Drinfeld module morphism $\epsilon:P_0\ra M$ such that
$$\cdots\ra P_{i+1}\ra P_i\ra \cdots\ra P_1\ra P_0\xra{\epsilon}M\ra 0$$
is an exact sequence in $\mc{YD}^H_H$.

\end{defn}

If each $P_i$, $i\le 0$, is a finitely generated free (resp. relative projective) Yetter-Drinfeld module, we call this complex $\textbf{P}_*$ a  finitely generated free (resp. relative projective) Yetter-Drinfeld module resolution.

Of course each free Yetter-Drinfeld module resolution is a free resolution and each relative projective Yetter-Drinfeld module resolution is a projective resolution.

\begin{lem}\label{lem resyd}
Let $\mc{C}$ be a cogroupoid and $X,Y\in\ob(\mc{C})$. The equivalent functor $-\sq_{\mc{C}(X,X)}\mc{C}(X,Y)$ sends any relative projective Yetter-Drinfeld module resolution $\textbf{P}_*$ of the trivial Yetter-Drinfeld module $\kk$ over $\mc{C}(X,X)$ to a  relative projective Yetter-Drinfeld module resolution $\textbf{P}_*\sq_{\mc{C}(X,X)}\mc{C}(X,Y)$ of the trivial Yetter-Drinfeld module $\kk$  over $\mc{C}(Y,Y)$. In particular, if $\textbf{P}_*$ is finitely generated (resp. bounded), then $\textbf{P}_*\sq_{\mc{C}(X,X)}\mc{C}(X,Y)$ is also finite generated (resp. bounded).
\end{lem}
\proof  By applying the functor $-\sq_{\mc{C}(X,X)}\mc{C}(X,Y)$ to the complex $\textbf{P}_*\ra \kk \ra 0$, we obtain the exact sequence of Yetter-Drinfeld modules
\begin{equation}\label{eq 2}\cdots \ra P_i\sq_{\mc{C}(X,X)}\mc{C}(X,Y)\xra{\dt_i\sq\mc{C}(X,Y)} P_{i-1}\sq_{\mc{C}(X,X)}\mc{C}(X,Y)\ra \cdots\hspace{12mm}\end{equation}
$$\hspace{10mm}\ra P_1\sq_{\mc{C}(X,X)}\mc{C}(X,Y)\ra P_0\sq_{\mc{C}(X,X)}\mc{C}(X,Y)\ra \kk\sq_{\mc{C}(X,X)}\mc{C}(X,Y)\ra 0.$$
It is easy to check that $\kk\sq_{\mc{C}(X,X)}\mc{C}(X,Y)\cong \kk$ as Yetter-Drinfeld modules over $\mc{C}(Y,Y)$. We claim that each $P_i\sq_{\mc{C}(X,X)}\mc{C}(X,Y)$  is a direct summand of a free Yetter-Drinfeld module.  Each $P_i$ is a  relative projective Yetter-Drinfeld module over $\mc{C}(X,X)$. That is, there is a Yetter-drinfeld module $Q_i$ and a comodule $V_i$, such that
$$P_i\op Q_i\cong V_i\bxt \mc{C}(X,X).$$ After applying the functor
$-\sq_{\mc{C}(X,X)}\mc{C}(X,Y)$, we obtain that
$$(P_i\op Q_i)\sq_{\mc{C}(X,X)}\mc{C}(X,Y)\cong (V_i\bxt \mc{C}(X,X))\sq_{\mc{C}(X,X)}\mc{C}(X,Y).$$
The cotensor functor $-\sq_{\mc{C}(X,X)}\mc{C}(X,Y)$ commutes with direct sums, so
$$(P_i\op Q_i)\sq_{\mc{C}(X,X)}\mc{C}(X,Y)\cong (P_i\sq_{\mc{C}(X,X)}\mc{C}(X,Y))\op( Q_i\sq_{\mc{C}(X,X)}\mc{C}(X,Y)).$$
Theorem 4.4 in \cite{bi} shows that
$$(V_i\bxt \mc{C}(X,X))\sq_{\mc{C}(X,X)}\mc{C}(X,Y)\cong (V_i\sq_{\mc{C}(X,X)}\mc{C}(X,Y))\bxt \mc{C}(Y,Y)$$ as Yetter-Drinfeld modules over $\mc{C}(Y,Y)$.
Therefore, we obtain the Yetter-Drinfeld module isomorphism
$$(P_i\sq_{\mc{C}(X,X)}\mc{C}(X,Y))\op( Q_i\sq_{\mc{C}(X,X)}\mc{C}(X,Y))\cong (V_i\sq_{\mc{C}(X,X)}\mc{C}(X,Y))\bxt \mc{C}(Y,Y).$$
Hence, each $P_i\sq_{\mc{C}(X,X)}\mc{C}(X,Y)$ is a relative projective Yetter-Drinfeld module over $\mc{C}(Y,Y)$. So $\textbf{P}_*\sq_{\mc{C}(X,X)}\mc{C}(X,Y)$ is a relative projective Yetter-Drinfeld module resolution of the trivial Yetter-Drinfeld module $\kk$  over $\mc{C}(Y,Y)$.

By \cite[Proposition 1.16]{bi1},  if $V_i$ is a finite dimensional comodule over $\mc{C}(X,X)$, then $V_i\sq_{\mc{C}(X,X)}\mc{C}(X,Y)$ is a  finite dimensional comodule over $\mc{C}(Y,Y)$. so if $\textbf{P}_*$ is finitely generated, then $\textbf{P}_*\sq_{\mc{C}(X,X)}\mc{C}(X,Y)$ is also finite generated. The argument for boundedness is clear.\qed

\subsection{Homological properties of cogroupoids}
From now on, until the end of the paper, we assume that the Hopf algebras mensioned have \textit{bijective} antipodes. we also assume that any cogroupoid $\mc{C}$ mentioned  satisfies that $S_{X,Y}$ is \textit{bijective} for any $X,Y\in \ob(\mc{C})$. This assumption is to make sure that $S_{Y,X}\circ S_{X,Y}$ is an algebra automorphism of $\mc{C}(X,Y)$. Actually, if $\mc{C}$ is a connected cogroupoid such that for some object $X$, $\mc{C}(X,X)$ is a Hopf algebra with bijective antipode, then $S_{X,Y}$ is bijective for any objects $X, Y$ (see Remark 2.6 in \cite{yu}).

Let $\mc{C}$ be a cogroupoid and $X,Y\in \ob(\mc{C})$. Both the morphisms $\bdt_{X,X}^Y:\mc{C}(X,X)\ra \mc{C}(X,Y)\ot \mc{C}(Y,X)$  and $S_{Y,X}:\mc{C}(Y,X)\ra \mc{C}(X,Y)^{op}$ are algebra homomorphisms, so
\begin{equation}\label{mod}D_1=(\id\ot S_{Y,X})\circ(\bdt^Y_{X,X}):\mc{C}(X,X)\ra \mc{C}(X,Y)^e\end{equation}
is an algebra homomorphism. This induces a functor $\mc{L}_X:\Mod$-$\mc{C}(X,Y)^e\ra \Mod$-$\mc{C}(X,X)$.  Let $M$ be a $\mc{C}(X,Y)$-bimodule.  The left $\mc{C}(X,X)$-module structure of $\mc{L}_X(M)$ is given by  $$x\ra m=x^{X,Y}_1mS_{Y,X}(x^{Y,X}_2),$$
for any $m\in M$ and $x\in \mc{C}(X,X)$.
The functor $\mc{L}_X$ is just  the functor $\mc{L}$ defined in \cite{yu}. We need to define the following functors. They share similar properties as the functor $\mc{L}_X$.
\begin{itemize}
\item The functor $\mc{R}_X:\Mod$-$\mc{C}(X,Y)^e\ra \Mod$-$\mc{C}(X,X)^{op}$ induced by the algebra homomorphism $D_2=(\id\ot S^{-1}_{X,Y})\circ(\bdt^Y_{X,X}):\mc{C}(X,X)\ra \mc{C}(X,Y)^e$;
\item The functor $\mc{R}_Y:\Mod$-$\mc{C}(X,Y)^e\ra \Mod$-$\mc{C}(Y,Y)^{op}$ induced by the algebra homomorphism $D_3=\tau\circ(S_{Y,X}\ot \id)\circ(\bdt^X_{Y,Y}):\mc{C}(Y,Y)\ra \mc{C}(X,Y)^e$;
\item The functor $\mc{L}_Y:\Mod$-$\mc{C}(X,Y)^e\ra \Mod$-$\mc{C}(Y,Y)$ induced by the algebra homomorphism $D_4=\tau\circ(S^{-1}_{X,Y}\ot \id)\circ(\bdt^X_{Y,Y}):\mc{C}(Y,Y)\ra \mc{C}(X,Y)^e$.
\end{itemize}
Here, $\tau:\mc{C}(X,Y)^{op}\ot \mc{C}(X,Y)\ra \mc{C}(X,Y)\ot \mc{C}(X,Y)^{op}$ is the flip map.

Let $M$ be a $\mc{C}(X,Y)$-bimodule. The right $\mc{C}(X,X)$-module structure of $\mc{R}_X(M)$, the right $\mc{C}(Y,Y)$-module structure of $\mc{R}_Y(M)$ and the left $\mc{C}(Y,Y)$-module structure of $\mc{L}_Y(M)$ is given by
$$m \leftarrow x=S^{-1}_{X,Y}(x^{Y,X}_2)m x^{X,Y}_1,$$
 $$m \leftarrow y=S_{Y,X}(x^{Y,X}_1)m x^{X,Y}_2,$$
and
$$y\ra m=x^{X,Y}_2mS^{-1}_{X,Y}(x^{Y,X}_1),$$
for any $m\in M$, $x\in \mc{C}(X,X)$ and $y\in \mc{C}(Y,Y)$, respectively.

As usual, we view $\mc{C}(X,Y)^e$ as a left and a right $\mc{C}(X,Y)^e$-module respectively in the following ways:
\begin{equation}(a\ot b)\ra (x\ot y)=a  x\ot y  b,\end{equation}
and \begin{equation}\label{lm}(x\ot y)\la(a\ot b) =x a\ot b  y,\end{equation}
for any $x\ot y$ and $a\ot b \in \mc{C}(X,Y)^e$. Then we have the  modules $\mc{L}_X(\mc{C}(X,Y)^e)$,  $\mc{R}_X(\mc{C}(X,Y)^e)$, $\mc{R}_Y(\mc{C}(X,Y)^e)$ and $\mc{L}_Y(\mc{C}(X,Y)^e)$. They are all free modules. Here we point out that we use the left $\mc{C}(X,Y)^e$-module for $\mc{L}$ and right $\mc{C}(X,Y)^e$-module for $\mc{R}$.

Let ${}_*\mc{C}(X,X)\ot \mc{C}(X,Y)$ be the left $\mc{C}(X,X)$-module defined by the left multiplication of the factor $\mc{C}(X,X)$, and
$\mc{C}(X,X)_*\ot \mc{C}(X,Y)$ be the right $\mc{C}(X,X)$-module defined by the right multiplication of the factor $\mc{C}(X,X)$. Similarly, let ${}_*\mc{C}(Y,Y)\ot \mc{C}(X,Y)$ be the left $\mc{C}(Y,Y)$-module defined by the left multiplication of the factor $\mc{C}(Y,Y)$, and
$\mc{C}(Y,Y)_*\ot \mc{C}(X,Y)$ be the right $\mc{C}(Y,Y)$-module defined by the right multiplication of the factor $\mc{C}(Y,Y)$.
Lemma 2.1 in \cite{yu} shows that  $\mc{L}_X(\mc{C}(X,Y)^e)\cong {}_*\mc{C}(X,X)\ot \mc{C}(X,Y)$ as left $\mc{C}(X,X)$-modules. The isomorphism is given as follows:
$$\begin{array}{ccl}\mc{L}_X(\mc{C}(X,Y)^e) &\lra &{}_*\mc{C}(X,X)\ot \mc{C}(X,Y)\\x\ot y &\longmapsto& x_1^{X,X} \ot y S_{Y,X}(S_{X,Y}(x_2^{X,Y}))\end{array}$$
with inverse
$$\begin{array}{rcl}
 {}_*\mc{C}(X,X)\ot \mc{C}(X,Y) &\lra &\mc{L}_X(\mc{C}(X,Y)^e)\\
x\ot y &\longmapsto  & x^{X,Y}_1 \ot y S_{Y,X}(x^{Y,X}_2).
\end{array}$$
Similarly, we obtain that
\begin{itemize}
\item $\mc{R}_X(\mc{C}(X,Y)^e)\cong  \mc{C}(X,X)_*\ot \mc{C}(X,Y)$ as right $\mc{C}(X,X)$-modules. The isomorphism is given by
$$\begin{array}{ccl}\mc{R}_X(\mc{C}(X,Y)^e) &\lra &\mc{C}(X,X)_*\ot \mc{C}(X,Y)\\x\ot y &\longmapsto& x_1^{X,X} \ot S^{-1}_{X,Y}(S^{-1}_{Y,X}(x_2^{X,Y})) y, \end{array}$$
with inverse
$$\begin{array}{rcl}
 \mc{C}(X,X)_*\ot \mc{C}(X,Y) &\lra &\mc{R}_X(\mc{C}(X,Y)^e)\\
x\ot y &\longmapsto  & x^{X,Y}_1 \ot S^{-1}_{X,Y}(x^{Y,X}_2) y;
\end{array}$$
\item $\mc{R}_Y(\mc{C}(X,Y)^e)\cong \mc{C}(Y,Y)_*\ot \mc{C}(X,Y)$  as right $\mc{C}(Y,Y)$-modules. The isomorphism is given by
$$\begin{array}{ccl}\mc{R}_Y(\mc{C}(X,Y)^e) &\lra &\mc{C}(Y,Y)_*\ot \mc{C}(X,Y)\\x\ot y &\longmapsto& x_2^{Y,Y} \ot S _{Y,X}(S _{X,Y}(x_1^{X,Y})) y,\end{array}$$
with inverse
$$\begin{array}{rcl}
 \mc{C}(Y,Y)_*\ot \mc{C}(X,Y) &\lra &\mc{R}_X(\mc{C}(X,Y)^e)\\
x\ot y &\longmapsto  & x^{X,Y}_2 \ot S_{Y,X}(x^{Y,X}_1)y;
\end{array}$$
\item $\mc{L}_Y(\mc{C}(X,Y)^e)\cong  {}_*\mc{C}(Y,Y)\ot \mc{C}(X,Y)$ as left $\mc{C}(Y,Y)$-modules. The isomorphism is given by
 $$\begin{array}{ccl}\mc{L}_Y(\mc{C}(X,Y)^e) &\lra &{}_*\mc{C}(Y,Y)\ot \mc{C}(X,Y)\\x\ot y &\longmapsto& x_2^{Y,Y} \ot y S^{-1}_{X,Y}(S^{-1}_{Y,X}(x_1^{X,Y})),\end{array}$$
with inverse
$$\begin{array}{rcl}
 {}_*\mc{C}(Y,Y)\ot \mc{C}(X,Y) &\lra &\mc{L}_Y(\mc{C}(X,Y)^e)\\
x\ot y &\longmapsto  & x^{X,Y}_2 \ot y S^{-1}_{X,Y}(x^{Y,X}_1).
\end{array}$$
\end{itemize}

It is showed in \cite[Lemma 2.2]{yu}  that the Hochschild cohomology of a bimodule $M$ over $\mc{C}(X,Y)$ can be computed through the extension groups of the trivial module ${}_\vps\kk$ by $\mc{L}_X(M)$. Similar results hold for the functors $\mc{R}_X$, $\mc{R}_Y$, $\mc{L}_Y$.

\begin{lem}\label{lem kk}Let $\mc{C}$ be a cogroupoid, and $X,Y\in \ob(\mc{C})$.  Let $M$ be a $\mc{C}(X,Y)$-bimodule.
\begin{enumerate}
\item[(i)] $\Ext^i_{\mc{C}(X,Y)^e}(\mc{C}(X,Y),M)\cong \Ext^i_{\mc{C}(X,X)}({}_\vps\kk,\mc{L}_X(M))$,
for all $i\le 0$.
\item[(ii)] $\Ext^i_{\mc{C}(X,Y)^e}(\mc{C}(X,Y),M)\cong \Ext^i_{\mc{C}(X,X)^{op}}(\kk_\vps,\mc{R}_X(M))$,
for all $i\le 0$.
\item[(iii)] $\Ext^i_{\mc{C}(X,Y)^e}(\mc{C}(X,Y),M)\cong \Ext^i_{\mc{C}(Y,Y)}({}_\vps\kk,\mc{L}_Y(M))$,
for all $i\le 0$.
\item[(iv)] $\Ext^i_{\mc{C}(X,Y)^e}(\mc{C}(X,Y),M)\cong \Ext^i_{\mc{C}(Y,Y)^{op}}(\kk_\vps,\mc{R}_Y(M))$,
for all $i\le 0$.
\end{enumerate}
\end{lem}

\subsection{Main results}
In order to state our main results we need to define winding automorphisms of cogroupoids.

Let $\mc{C}$ be a cogroupoid and $X,Y\in \ob(\mc{C})$. Let $\xi:\mc{C}(X,X)\ra \kk$ be an algebra homomorphism.
The \textit{left winding automorphism} $[\xi]_{X,Y}^l$ of  $\mc{C}(X,Y)$ associated to  $\xi$ is defined to be
$$[\xi]_{X,Y}^l(a^{X,Y})=\xi(a_1^{X,X})a_2^{X,Y},$$
for any $a\in \mc{C}(X,Y)$.
Let $\eta:\mc{C}(Y,Y)\ra \kk$ be an algebra homomorphism. Similarly, the \textit{right winding automorphism} of  $\mc{C}(X,Y)$ associated to $\eta$ is defined to be $$[\eta]_{X,Y}^r(a^{X,Y})=a_1^{X,Y}\eta(a_2^{Y,Y}),$$
for any $a\in \mc{C}(X,Y)$.

\begin{lem}\label{windingaut}
Let $\mc{C}$ be a cogroupoid and  $X,Y\in \ob(\mc{C})$, let $\xi:\mc{C}(X,X)\ra \kk$, and $\eta:\mc{C}(Y,Y)\ra \kk$ be algebra homomorphisms.
\begin{enumerate}
\item[(i)] $([\xi]_{X,X}^l)^{-1}=[\xi S_{X,X}]^l$.
\item[(ii)] $\xi S_{X,X}^2=\xi$, so $[\xi]_{X,X}^l=[\xi S_{X,X}^2]_{X,X}^l$.
\item[(iii)] $[\xi]_{X,Y}^l\circ S_{Y,X}\circ S_{X,Y}=S_{Y,X}\circ S_{X,Y}\circ [\xi]_{X,Y}^l$.
\item[(i')] $([\eta]_{Y,Y}^r)^{-1}=[\eta S_{Y,Y}]^r$.
\item[(ii')] $\eta S_{Y,Y}^2=\eta$, so $[\eta]_{Y,Y}^r=[\eta S_{Y,Y}^2]_{Y,Y}^r$.
\item[(iii')] $[\eta]_{X,Y}^r\circ S_{Y,X}\circ S_{X,Y}=S_{Y,X}\circ S_{X,Y}\circ [\eta]_{X,Y}^r$.
\end{enumerate}
\end{lem}
\proof Since $\mc{C}(X,X)$ is a Hopf algebra, (i) and (ii) are just Lemma 2.5 in \cite{bz}. (i') and (ii') hold similarly.
We only need to prove (iii), and (iii') can be proved similarly.

For $x\in \mc{C}(X,Y)$,
$$S_{Y,X}\circ S_{X,Y}\circ [\xi]_{X,Y}^l(a^{X,Y})=\xi(a_1^{X,X})S_{Y,X}( S_{X,Y}(a_2^{X,Y})).$$
Since $\bdt_{X,Y}^X(S_{Y,X}( S_{X,Y}(a^{X,Y})))=S^2_{X,X}(a_1^{X,X})\ot S_{Y,X}( S_{X,Y}(a_2^{X,Y}))$,
$$[\xi]_{X,Y}^l\circ S_{Y,X}\circ S_{X,Y}(a^{X,Y})=\xi{S^2_{X,X}}(a_1^{X,X})S_{Y,X}( S_{X,Y}(a_2^{X,Y}))$$
By (ii),   $\xi S_{X,X}^2=\xi$, so
$$S_{Y,X}\circ S_{X,Y}\circ [\xi]^l(a^{X,Y})=[\xi]^l\circ S_{Y,X}\circ S_{X,Y}(a^{X,Y}).$$
Therefore, $S_{Y,X}\circ S_{X,Y}\circ [\xi]_{X,Y}^l=[\xi]_{X,Y}^l\circ S_{Y,X}\circ S_{X,Y}$. \qed

The following is the main result of \cite{yu}.
\begin{prop}\label{prop ga}
Let $\mc{C}$ be a connected  cogroupoid and let $X\in \ob(\mc{C})$ such that $\mc{C}(X,X)$ is a twisted CY algebra of dimension $d$ with left homological integral $\int^l_{\mc{C}(X,X)}=\kk_\xi$,  where $\xi:\mc{C}(X,X)\ra \kk$ is an algebra homomorphism. Then for any $Y\in \ob(\mc{C})$, $\mc{C}(X,Y)$ is a twisted CY algebra of dimension $d$ with Nakayama automorphism $\mu$ defined as $\mu=S_{Y,X}\circ S_{X,Y}\circ [\xi]_{X,Y}^l$. That is,
$$\mu(a)=\xi(a_1^{X,X})S_{Y,X}(S_{X,Y}(a_2^{X,Y})),$$
for any $x\in \mc{C}(X,Y)$.
\end{prop}

Though we do not say that the CY-dimension of $\mc{C}(X,X)$ and $\mc{C}(X,Y)$ are same in the statement of \cite[Theorem 2.5]{yu}, it is easy to see from its proof.

Let $\mc{C}$ be a cogroupoid. We define a cogroupoid $\mc{C}'$ as follows:
\begin{itemize}
\item $\ob(\mc{C}')=\ob(\mc{C})$.
\item For any objects $Y$, $X$, the algebra $\mc{C}'(Y,X)$ is the algebra $\mc{C}(X,Y)$.
\item For any objects $Y$, $X$ and $Z$, the algebra homomorphism $\bdt'^Z_{YX}:\mc{C}'(Y,X)\ra \mc{C}'(Y,Z)\ot \mc{C}'(Z,X)$ is the algebra homomorphism $\tau\circ\bdt_{XY}^Z:\mc{C}(X,Y)\ra \mc{C}(Z,Y)\ot \mc{C}(X,Z)$ in $\mc{C}$, where $\tau:\mc{C}(X,Z)\ot \mc{C}(Z,Y)\ra \mc{C}(Z,Y)\ot \mc{C}(X,Z)$ is the flip map.
\item For any object $X$, $\vps'_X:\mc{C}'(X,X)\ra \kk$ is the same as $\vps_X:\mc{C}(X,X)\ra \kk$ in $\mc{C}$.
\item For any objects $Y$, $X$, $S'_{Y,X}:\mc{C}'(Y,X)\ra \mc{C}'(X,Y)$ is the morphism $S^{-1}_{Y,X}:\mc{C}(X,Y)\ra \mc{C}(Y,X)$.
\end{itemize}
It is easy to check that this indeed defines a cogroupoid. Now apply Proposition \ref{prop ga} to the cogroupoid $\mc{C}'$, we obtain the following corollary.

\begin{cor}\label{cor ga}
Let $\mc{C}$ be a connected cogroupoid and let $Y\in \ob(\mc{C})$ such that $\mc{C}(Y,Y)$ is a twisted CY algebra of dimension $d$ with left homological integral $\int^l_{\mc{C}(Y,Y)}=\kk_\eta$,  where $\eta:\mc{C}(Y,Y)\ra \kk$ is an algebra homomorphism. Then for any $X\in \ob(\mc{C})$, $\mc{C}(X,Y)$ is a twisted CY algebra of  dimension $d$ with Nakayama automorphism $\mu'$ defined as $\mu'=S^{-1}_{X,Y}\circ S^{-1}_{Y,X}\circ [\eta]_{X,Y}^r$. That is,
$$\mu'(a)=S^{-1}_{X,Y}(S^{-1}_{Y,X}(a_1^{X,Y}))\eta(a_2^{Y,Y}),$$
for any $x\in \mc{C}(X,Y)$.
\end{cor}

\begin{thm}\label{thm main}
Let $\mc{C}$ be a connected cogroupoid and let  $X$ be an object in $\mc{C}$ such that $\mc{C}(X,X)$ is a twisted CY Hopf algebra of dimension $d$. Then for any $Y\in\ob(\mc{C})$ such that $\mc{C}(Y,Y)$ is homologically smooth,  $\mc{C}(Y,Y)$ is a twisted CY algebra of dimension $d$ as well.
\end{thm}
\proof Let $Y$ be an object in $\mc{C}$ such that $\mc{C}(Y,Y)$ is homologically smooth. We need to compute
 the Hochschild cohomology of $\mc{C}(X,Y)$. By Lemma \ref{lem kk},
$$\Ext^i_{\mc{C}(X,Y)^e}(\mc{C}(X,Y),\mc{C}(X,Y)^e)\cong\Ext^i_{\mc{C}(Y,Y)^{op}}(\kk_\vps,\mc{R}_Y(\mc{C}(X,Y)^e))$$
for all $i\le 0$. $\mc{R}_Y(\mc{C}(X,Y)^e)$ is a $\mc{C}(X,Y)^e$-$\mc{C}(Y,Y)$-bimodule. The right $\mc{C}(Y,Y)$-module isomorphism
$$\begin{array}{ccl}\mc{R}_Y(\mc{C}(X,Y)^e) &\lra &\mc{C}(Y,Y)_*\ot \mc{C}(X,Y)\\x\ot y &\longmapsto& x_2^{Y,Y} \ot S _{Y,X}(S _{X,Y}(x_1^{X,Y})) y,\end{array}$$
is also an isomorphism of left $\mc{C}(X,Y)^e$-modules if we endow a left $\mc{C}(X,Y)^e$-module structure on $\mc{C}(Y,Y)_*\ot \mc{C}(X,Y)$ as follows:
$$(a\ot b)\ra (x\ot y)=a^{Y,Y}_2x\ot S_{Y,X}(S_{X,Y}(a_1^{X,Y}))yb,$$
for any   $x\ot y\in \mc{C}(Y,Y)\ot \mc{C}(X,Y)$ and $a\ot b\in \mc{C}(X,Y)^e$.
Therefore, we obtain  the following left $\mc{C}(X,Y)^e$-module isomorphisms:
$$\begin{array}{rcl}\Ext^i_{\mc{C}(X,Y)^e}(\mc{C}(X,Y),\mc{C}(X,Y)^e)&\cong&\Ext^i_{\mc{C}(Y,Y)^{op}}(\kk_\vps,\mc{R}_Y(\mc{C}(X,Y)^e))\\
&\cong& \Ext^i_{\mc{C}(Y,Y)^{op}}(\kk_\vps,\mc{C}(Y,Y)_*\ot \mc{C}(X,Y))\\
&\cong&\Ext^i_{\mc{C}(Y,Y)^{op}}(\kk_\vps,\mc{C}(Y,Y))\ot \mc{C}(X,Y)
\end{array}$$
for $i\le 0$.  The left $\mc{C}(X,Y)^e$-module structure on $\Ext^i_{\mc{C}(Y,Y)^{op}}(\kk_\vps,\mc{C}(Y,Y))\ot \mc{C}(X,Y)$ induced by the isomorphisms above is given by
$$(a\ot b)\ra (x\ot y)=a^{Y,Y}_2x\ot S_{Y,X}(S_{X,Y}(a_1^{X,Y}))yb,$$
for any   $x\ot y\in \Ext^i_{\mc{C}(Y,Y)^{op}}(\kk_\vps,\mc{C}(Y,Y))\ot \mc{C}(X,Y)$ and $a\ot b\in \mc{C}(X,Y)^e$.
Note that the left $\mc{C}(Y,Y)$-module structure of $\mc{C}(Y,Y)$ induces a left $\mc{C}(Y,Y)$-module structure on $\Ext^i_{\mc{C}(Y,Y)^{op}}(\kk_\vps,\mc{C}(Y,Y))$.

It follows from Proposition \ref{prop ga}  that $\mc{C}(X,Y)$ is a twisted CY algebra of dimension $d$ with Nakayama automorphism $\mu=S_{Y,X}\circ S_{X,Y}\circ [\xi]_{X,Y}^l$. So
 $$\Ext^i_{\mc{C}(X,Y)^e}(\mc{C}(X,Y),\mc{C}(X,Y)^e)=\begin{cases}0&i\neq d;\\\mc{C}(X,Y)^\mu&i=d.\end{cases}$$
Now we arrive at the isomorphism of left $\mc{C}(X,Y)^e$-modules
$$\Ext^i_{\mc{C}(Y,Y)^{op}}(\kk_\vps,\mc{C}(Y,Y))\ot \mc{C}(X,Y)\cong \begin{cases}0&i\neq d;\\
\mc{C}(X,Y)^\mu&i=d.\end{cases}$$
A left  $\mc{C}(X,Y)^e$-module  can be viewed as a $\mc{C}(X,Y)$-bimodule. The right module structure of $\Ext^i_{\mc{C}(Y,Y)^{op}}(\kk_\vps,\mc{C}(Y,Y))\ot \mc{C}(X,Y)$ is just the right multiplication to the factor $\mc{C}(X,Y)$. So especially, as right $\mc{C}(X,Y)$-modules,
$$\Ext^i_{\mc{C}(Y,Y)^{op}}(\kk_\vps,\mc{C}(Y,Y))\ot \mc{C}(X,Y)\cong \begin{cases}0&i\neq d;\\
\mc{C}(X,Y)&i=d.\end{cases}$$
This shows that $\Ext^i_{\mc{C}(Y,Y)^{op}}(\kk_\vps,\mc{C}(Y,Y))=0$ for $i\neq d$. Moreover, for degree $d$, we denote $V=\Ext^d_{\mc{C}(Y,Y)^{op}}(\kk_\vps,\mc{C}(Y,Y))$. Then $V\otimes C(X,Y)\cong C(X,Y)$ as free right $C(X,Y)$-modules. Hence $0<\dim V<\infty$ (note that we do not know whether $C(X,Y)$ has the FBN property). Similarly, $\Ext^i_{\mc{C}(Y,Y)}({}_\vps\kk,\mc{C}(Y,Y))=0$ for $i\neq d$ and $\Ext^d_{\mc{C}(Y,Y)}({}_\vps\kk,\mc{C}(Y,Y))$ is finite dimensional as well. Hence $\mc{C}(Y,Y)$ is twisted CY of dimension $d$ by Proposition \ref{prop ascy}. \qed

\begin{thm}\label{thm main'}
Let $H$ and $L$ be two  monoidally Morita-Takeuchi equivalent Hopf algebras. If  $H$ is twisted CY of dimension $d$ and $L$ is homologically smooth, then $L$ is twisted CY of dimension $d$ as well.
\end{thm}
\proof This directly follows from Theorem \ref{thm ga} and Theorem \ref{thm main}.

Before we present our main theotem, we need  the following lemma.
\begin{lem}\label{lem res}
Let $H$ be a Noetherian Hopf algebra. Then the trivial Yetter-Drinfeld module $\kk$ admits a finitely generated free Yetter-Drinfeld module resolution.
\end{lem}
\proof First we have an epimorphism $\epsilon:\kk\bxt H\ra \kk$, $1\ot h\mapsto \vps(h)$ of Yetter-Drinfeld modules. Set $P_0=\kk\bxt H$. Since $H$ is Noetherian, $\Ker \epsilon$ is finitely generated as a module over $H$. Say it is generated by a finite dimensional subspace $V_1$ of $P_0$. That is, there exists an epimorphism $V_1\ot H\ra \Ker \epsilon\ra 0$ given by $v\otimes h\mapsto vh$ for any $v\in V_1$ and $h\in H$. Let $C_1$ be the subcomodule of $\Ker \epsilon$ generated by $V_1$. We know $C_1$ is finite dimensional since $V_1$ is finite dimensional by the fundamental theory of comodules. Construct the epimorphism $C_1\bxt H\ra \Ker \epsilon\ra 0$ via $c\ot h\mapsto ch$ for any $c\in C_1$ and $h\in H$. It is easy to check that it is a morphism of Yetter-Drinfeld modules.  Set $P_1=C_1\bxt H$, we have the exact sequence
$P_1\ra P_0\ra \kk\ra0$. Note that $P_1$ is again a Noetherian $H$-module. Hence we can do the procedure recursively to obtain  a finitely generated free Yetter-Drinfeld module resolution of $\kk$. \qed

\begin{thm}\label{thm main1}
Let $H$ be a twisted CY Hopf algebra of dimension $d$, and  $L$  a Hopf algebra monoidally Morita-Takeuchi equivalent to $H$.  If one of the following conditions holds,  then $L$ is also twisted CY of dimension $d$.
\begin{enumerate}
\item[(i)] $H$ admits a finitely generated relative projective Yetter-Drinfeld module resolution for the trivial Yetter-Drinfeld module $\kk$ and $L$ has finite global dimension.
\item[(ii)] $H$ admits a bounded finitely generated relative projective Yetter-Drinfeld module resolution for the trivial Yetter-Drinfeld module $\kk$.
\item[(iii)] $H$ is Noetherian and $L$ has finite global dimension.
\item[(iv)] $L$ is Noetherian and has finite global dimension.
\end{enumerate}
\end{thm}
\proof By Theorem \ref{thm main}, we only need to prove that if one of the conditions listed in the Theorem holds, then $L$ is homologically smooth.

(i) We use the language of cogroupoids. Since $H$ and $L$ are monoidally Morita-Takeuchi equivalent, there exists a connected cogroupoid with 2 objects $X,Y$ such that $H=\mc{C}(X,X)$ and $L=\mc{C}(Y,Y)$ (Theorem \ref{thm ga}). By Lemma \ref{prop hom}, to show $L=\mc{C}(Y,Y)$ is homologically smooth,  we only need to show that the trivial module $\kk_\vps$ admits a bounded projective resolution with each term finitely generated. By assumption, the trivial Yetter-Drinfeld module $\kk$ over the Hopf algebra $H=\mc{C}(X,X)$ admits a finitely generated relative projective Yetter-Drinfeld module resolution
\begin{equation}\label{eq 3}\cdots \ra P_i\xra{\dt_i} P_{i-1}\ra \cdots\ra P_1\ra P_0\ra \kk\ra 0.\end{equation}

By Lemma \ref{lem resyd},
\begin{equation}\label{eq 6}\cdots \ra P_i\sq_{\mc{C}(X,X)}\mc{C}(X,Y)\xra{\dt_i\sq\mc{C}(X,Y)} P_{i-1}\sq_{\mc{C}(X,X)}\mc{C}(X,Y)\ra \cdots\hspace{12mm}\end{equation}
$$\hspace{18mm}\ra P_1\sq_{\mc{C}(X,X)}\mc{C}(X,Y)\ra P_0\sq_{\mc{C}(X,X)}\mc{C}(X,Y)\ra \kk\ra 0.$$
is a finitely generated relative projective Yetter-Drinfeld module resolution of the trivial Yetter-Drinfeld module $\kk$ over  $\mc{C}(Y,Y)$.
Hence, each $P_i\sq_{\mc{C}(X,X)}\mc{C}(X,Y)$ is a finite generated projective $\mc{C}(Y,Y)$-module.  By assumption, the global dimension of $\mc{C}(Y,Y)$ is finite, say $n$. Set $K_n=\Ker(\dt_{n-1}\sq_{\mc{C}(X,X)}\mc{C}(X,Y))$. Following from Lemma 4.1.6 in \cite{we}, $K_n$ is projective, so it is a direct summand of $P_n\sq_{\mc{C}(X,X)}\mc{C}(X,Y)$. Since $P_n\sq_{\mc{C}(X,X)}\mc{C}(X,Y)$ is finitely generated, $K_n$ is finitely generated as well.
Therefore, $$0\ra K_n\ra P_{n-1}\sq_{\mc{C}(X,X)}\mc{C}(X,Y)\ra \cdots\hspace{25mm}$$
$$\hspace{10mm}\ra P_1\sq_{\mc{C}(X,X)}\mc{C}(X,Y)\ra P_0\sq_{\mc{C}(X,X)}\mc{C}(X,Y)\ra \kk\ra 0$$
is a bounded projective resolution with each term finitely generated. Hence, $L=\mc{C}(Y,Y)$ is homologically smooth.

(ii) It can be proved by using the similar argument in (i) since equations \eqref{eq 3} and \eqref{eq 6} now are bounded finitely generated projective resolutions for $\kk$.

%

(iii) It is a direct consequence of Lemma \ref{lem res} and (i).

(iv) The Hopf algebra $L$ is homologically smooth in this case follows from \cite[Lemma 5.2]{bz}. \qed

\begin{cor}\label{main cor}
Let $H$ and $L$ be two  monoidally Morita-Takeuchi equivalent Hopf algebras.  If both $H$ and $L$ are twisted CY, then $\gldim(H)=\gldim(L)$.
\end{cor}
\begin{proof}
It follows from Theorem \ref{thm main1} and the fact that for twisted CY Hopf algebras the CY dimension always equals the global dimension by Proposition \ref{prop ascy}.
\end{proof}

Now we discuss the relation between the homological integrals of $\mc{C}(X,X)$ and $\mc{C}(Y,Y)$ when both of them are twisted CY.
\begin{thm}\label{thm inn}
Let $\mc{C}$ be a connected cogroupoid. If $X$ and $Y$ are two objects such that $\mc{C}(X,X)$ and $\mc{C}(Y,Y)$ are both twisted CY algebras. Then we have
\begin{equation}\label{eq 8}(S_{Y,X}\circ S_{X,Y})^2=[\eta]_{X,Y}^r\circ ([\xi]_{X,Y}^l)^{-1}\circ \gamma,\end{equation}
where $\xi:\mc{C}(X,X)\ra \kk$  and $\eta:\mc{C}(Y,Y)\ra \kk$ are algebra homomorphisms given by the left homological integrals of $\mc{C}(X,X): \int^l_{\mc{C}(X,X)}=\kk_\xi$ and $\mc{C}(Y,Y): \int^l_{\mc{C}(Y,Y)}=\kk_\eta$ respectively, and $\gamma$ is an inner automorphism of $\mc{C}(X,Y)$.
\end{thm}
\proof From Proposition \ref{prop ga} and Corollary \ref{cor ga}, it is easy to see that the  CY-dimensions of $\mc{C}(X,X)$ and $\mc{C}(Y,Y)$ are equal. Moreover, the Nakayama automorphisms of $\mc{C}(X,Y)$ are given by $\mu=S_{Y,X}\circ S_{X,Y}\circ [\xi]^l$ and  $\mu'=S^{-1}_{X,Y}\circ S^{-1}_{Y,X}\circ [\eta]^r$. Since Nakayama automorphisms are unique up to inner automorphisms, thus
$$S_{Y,X}\circ S_{X,Y}\circ [\xi]_{X,Y}^l=S^{-1}_{X,Y}\circ S^{-1}_{Y,X}\circ [\eta]_{X,Y}^r\circ \gamma,$$
for some inner automorphism $\gamma$ of $\mc{C}(X,Y)$. Since $[\xi]_{X,Y}^l\circ S_{Y,X}\circ S_{X,Y}=S_{Y,X}\circ S_{X,Y}\circ [\xi]_{X,Y}^l$, we obtain that
$$(S_{Y,X}\circ S_{X,Y})^2=([\xi]_{X,Y}^l)^{-1}\circ [\eta]_{X,Y}^r\circ \gamma.$$
\qed

\begin{rem}

(i) We concentrate on CY property in this paper, but it is not hard to see that the above theorem holds when $\mc{C}(X,X)$ and $\mc{C}(Y,Y)$ are both AS-Gorenstein.

(ii) The three maps composed to give $(S_{Y,X}\circ S_{X,Y})^2$ in \eqref{eq 8} commute with each other. This can be proved as in \cite[Proposition 4.6]{bz} with the help of Lemma \ref{windingaut}. The equation (\ref{eq 8}) is just (4.6.1) in \cite{bz} when $X=Y$. One deduces at once the main result of \cite{ra}, that is the antipode $S$ has finite order when the Hopf algebra $H$ is finite dimensional. Since the inner automorphism $\gamma=(S_{Y,X}\circ S_{X,Y})^2\circ ([\eta]_{X,Y}^r)^{-1}\circ [\xi]_{X,Y}^l$ is intrinsic in $\mc{C}(X,Y)$, it prompts to generalize \cite[Question 4.6]{bz} to the Hopf-biGalois object $C(X,Y)$ when both $C(X,X)$ and $C(Y,Y)$ are AS-Gorenstein.
 \begin{question}
What is the inner automorphism in Theorem \ref{thm inn}?
\end{question}
\end{rem}

\section{Examples}\label{3}
In this section, we provide some examples.
\subsection{Example 1}
We take the field $\kk$ to be $\mathbb{C}$ in this subsection. Let $E\in \GL_m(\mathbb{C})$ with $m\le 2$ and let $\mc{B}(E)$ be the algebra  presented by generators $(u_{ij})_{1\se i,j\se m}$ and  relations
$$E^{-1}u^tEu=I_m=uE^{-1}u^tE,$$
where $u$ is the matrix $(u_{ij})_{1\se i,j\se m}$, $u^t$ is the transpose of $u$ and $I_m$ is the identity matrix. The algebra $\mc{B}(E)$ is a Hopf algebra and was defined by Dubois-Violette and Launer \cite{dvl} as the quantum automorphism group of the non-degenerate bilinear form associated to $E$. When
$$E=E_q=\left(\begin{array}{cc}0&1\\-q^{-1}&0\end{array}\right),$$
$\mc{B}(E_q)$ is just the algebra $\mc{O}_q(\SL_2(\mathbb{C}))$, which is the quantised coordinate algebra of $\SL_2(\mathbb{C})$.

In order to describe Hopf algebras whose comodule categories are monoidally equivalent to the one of $\mc{B}(E)$, we recall the cogroupoid $\mc{B}$.

Let $E\in \GL_m(\mathbb{C})$ and let  $F\in \GL_n(\mathbb{C})$. The algebra $\mc{B}(E,F)$ is defined to be  the algebra with generators $u_{ij}$, $1\se i\se m$, $1\se j\se n$, subject to the relations:
\begin{equation}\label{alg}F^{-1}u^t E u=I_n;\;\;uF^{-1}u^t E=I_m.\end{equation}
The generators $u_{ij}$ in $\mc{B}(E,F)$ is denoted by $u_{ij}^{EF}$ to express the dependence on $E$ and $F$ when needed.
It is clear that $\mc{B}(E)=\mc{B}(E,E)$.

For any $E\in \GL_m(\mathbb{C})$, $F\in \GL_n(\mathbb{C})$ and $G\in \GL_p(\mathbb{C})$, define the following maps:
\begin{equation}\label{bdt}\begin{array}{rcl}
\bdt_{E,F}^G:\mc{B}(E,F)&\longrightarrow &\mc{B}(E,G)\ot \mc{B}(G,F)\\
u_{ij}&\longmapsto &\sum_{k=1}^p u_{ik}\ot u_{kj},
\end{array}\end{equation}
\begin{equation}\label{bvps}\begin{array}{rcl}
\vps_E:\mc{B}(E)&\longrightarrow &\mathbb{C}\\
u_{ij}&\longmapsto &\dt_{ij},
\end{array}\end{equation}
\begin{equation}\label{bss}\begin{array}{rcl}
S_{E,F}:\mc{B}(E,F)&\longrightarrow &\mc{B}(F,E)^{op}\\
u&\longmapsto &E^{-1}u^tF.
\end{array}\end{equation}
It is clear that $S_{E,F}$ is bijective.

Lemma 3.2 in \cite{bi1} ensures that with these morphisms we have a cogroupoid. The cogroupoid $\mc{B}$  is defined as follows:
\begin{enumerate}
\item[(i)] $\ob(\mc{B})=\{E\in \GL_m(\mathbb{C}),m\le1\}$.
\item[(ii)] For $E,F\in \ob(\mc{B})$, the algebra $\mc{B}(E,F)$ is the algebra defined as in (\ref{alg}).
\item[(iii)] The structural maps $\bdt^\bullet_{\bullet,\bullet}$, $\vps_\bullet$ and $S_{\bullet,\bullet}$ are defined in (\ref{bdt}), (\ref{bvps}) and (\ref{bss}), respectively.
\end{enumerate}

\begin{lem}\rm{(\cite{bi2},\cite[Lemma 3.4]{bi1})}
Let $E\in \GL_m(\mathbb{C})$, $F\in \GL_n(\mathbb{C})$ with $m,n\le2$. Then $\mc{B}(E,F)\neq (0)$ if and only if $\tr(E^{-1}E^t)=\tr(F^{-1}F^t)$.
\end{lem}

This lemma induces the following corollary.
\begin{cor}
Let $\lambda\in \mathbb{C}$. Consider the full subcogroupoid $\mc{B}^\lmd$ of $\mc{B}$ with objects
$$\ob(\mc{B}^\lmd)=\{E\in \GL_n(\mathbb{C}),m\le2, \tr(E^{-1}E^t)=\lmd\}.$$
Then $\mc{B}^\lmd$ is a connected cogroupoid.
\end{cor}

Therefore, if $E\in \GL_m(\mathbb{C})$ and $F\in \GL_n(\mathbb{C})$ with $m,n\le2$ satisfy that $\tr(E^{-1}E^t)=\tr(F^{-1}F^t)$, then the comodule categories of $\mc{B}(E)$ and $\mc{B}(F)$ are monoidally equivalent. The results in \cite[Section 6]{bi} (cf. \cite{ww} and \cite{yu}) shows that both $\mc{B}(E)$ and $\mc{B}(F)$ are twisted CY algebras. Their left homological integrals $\int^l_{\mc{B}(E)}=\mathbb{C}_{\xi}$ and $\int^l_{\mc{B}(F)}=\mathbb{C}_{\eta}$ are given as $\xi_E(u^E)=(E^t)^{-1}E(E^t)^{-1}E$ and $\xi_F(u^F)=(F^t)^{-1}F(F^t)^{-1}F$, respectively. One checks that $\xi$ and $\eta$ satisfy the equation
$$\begin{array}{rcl}(S_{F,E}\circ S_{E,F})^2(u^{EF})&=&[\eta]_{E,F}^r\circ ([\xi]_{E,F}^l)^{-1}(u^{EF})\\
&=&E^{-1}E^tE^{-1}E^tu^{EF}F^{-1}F^tF^{-1}F^t.\end{array}$$
That is, the inner automorphism in Theorem \ref{thm inn} is just the identity. This is because that there are no nontrivial units in $\mc{B}(E,F)$.

In \cite{bi}, the Calabi-Yau property of $\mc{B}(E)$ is proved by the following steps:
\begin{enumerate}
\item[(i)] Construct a bounded complex $\textbf{P}_{*}(E)$ of finitely generated free Yetter-Drinfeld modules over $\mc{B}(E)$ for each $E\in \GL_m(\mathbb{C})$, $m\le2$. When it is exact, it is a free Yetter-Drinfeld module resolution of the trivial Yetter-Drinfeld module $\mathbb{C}$.
\item[(ii)] To show that the for $E\in \GL_m(\mathbb{C})$, $F\in \GL_n(\mathbb{C})$ with $\tr(E^{-1}E^t)=\tr(F^{-1}F^t)$ and $m,n\le2$, the complex $\textbf{P}_{*}(E)$ is exact if and only if $\textbf{P}_{*}(F)$ is exact.
\item[(iii)] Check that for any $q\in\mathbb{C}^\times$, the sequence $\textbf{P}_{*}(E_q)$ is exact. This is computable since $\textbf{P}_{*}(E_q)$ is a resolution of length three over a Noetherian algebra $\mc{O}_q(\SL_2(\mathbb{C}))$ with a nice PBW basis. For any $E\in \GL_m(\mathbb{C})$ with $m\le 2$. There is a $q\in\mathbb{C}^\times$ such that $\tr(E^{-1}E^t)=-q-q^{-1}=\tr(E_q^{-1}E_q^t)$, so $\textbf{P}_{*}(E)$ is exact.
\item[(iv)] Compute the extension group $\Ext^*_{\mc{B}(E)^{op}}(\mathbb{C}_\vps,\mc{B}(E))$ by the complex $\textbf{P}_{*}(E)$.
\end{enumerate}
By Theorem \ref{thm main}, this procedure can be simplified. We only need to prove that $\mc{B}(E_q)$ is twisted CY. As in Steps (i) and (iii), the trivial Yetter-Drinfeld module over $\mc{B}(E_q)$ admits a bounded finitely generated free Yetter-Drinfeld module resolution. From this resolution, we can conclude that  $\mc{B}(E_q)$ is a twisted CY algebra with left homological integral $\int^l_{\mc{B}(E_q)}=\mathbb{C}_\eta$ given by $$\eta(u)=\left(\begin{array}{cc}q^{-2}&0\\0&q^2\end{array}\right).$$
For any $E\in \GL_m(\mathbb{C})$ with $m\le 2$. There is a $q\in\mathbb{C}^\times$ such that $\tr(E^{-1}E^t)=-q-q^{-1}=\tr(E_q^{-1}E_q^t)$. So $\mc{B}(E)$ and  $\mc{B}(E_q)$ are monoidally Morita-Takeuchi equivalent.  The algebra $\mc{B}(E)$ is twisted CY by Theorem \ref{thm main'}.  Let  $\int^l_{\mc{B}(E)}={\mathbb{C}}_\xi$ be the left homological integral of $\mc{B}(E)$, where $\xi:\mc{B}(E)\ra \mathbb{C}$ is an algebra homomorphism. As we mentioned before, there are no nontrivial units in $\mc{B}(E,E_q)$. Then $\xi$ and $\eta$ satisfies the equation
 $$(S_{E_q,E}\circ S_{E,E_q})^2=[\eta]_{E,E_q}^r\circ ([\xi]_{E,E_q}^l)^{-1}.$$
So $\xi$ is defined by $\xi(u^E)=(E^t)^{-1}E(E^t)^{-1}E$.

\subsection{Example 2}
Before we present this example, we recall the definition of the 2-cocycle cogroupoid.

Let $H$ be a Hopf algebra with bijective antipode. A  \it{(right) 2-cocycle} on $H$ is a convolution invertible linear map $\sg:H\ot H\ra \kk$ satisfying
$$\label{rcocyle1}\sg(h_1,k_1)\sg(h_2k_2,l)=\sg(k_1,l_1)\sg(h,k_2l_2)$$
$$\label{rcocyle2}\sg(h,1)=\sg(1,h)=\vps(h)$$ for all $h,k,l\in H$. The set of 2-cocycles on $H$ is denoted $Z^2(H)$.
They defines the 2-cocycle cogroupoid of $H$.

Let $\sg,\tau\in Z^2(H)$. The algebra $H(\sg,\tau)$ is defined to be the vector space $H$ together with the multiplication  given by
\begin{equation}\label{mulco}x\centerdot y=\sg(x_1,y_1)x_2y_2\tau^{-1}(x_3,y_3),\end{equation}
for any $x,y\in H$.

The Hopf algebra $H(\sg,\sg)$ is just the \textit{cocycle deformation} $H^\sg$ of $H$ defined by Doi in \cite{do}. The comultiplication of $H^\sg$ is the same as the comultiplication of $H$. However, the multiplication and the antipode are deformed:
$$h\centerdot k=\sg(h_1,k_1)h_2k_2\sg^{-1}(h_3,k_3),$$
$$S_{\sg,\sg}(h)=\sg(h_1,S(h_2))S(h_3)\sg^{-1}(S(h_4),h_5)$$
for any $h,k\in H^\sg$.

Now we recall the necessary structural maps for the 2-cocycle cogroupoid of $H$. For any $\sg,\tau, \om\in Z^2(H)$, define the following maps:
\begin{equation}\label{comulco}\begin{array}{rcl}\bdt_{\sg,\tau}^\om=\bdt:H(\sg,\tau)&\longrightarrow &H(\sg,\om)\ot H(\om,\tau)\\
x&\longmapsto &x_1\ot x_2.
 \end{array}\end{equation}
\begin{equation}\label{couco}\vps_\sg=\vps:H(\sg,\sg)\longrightarrow \kk.\end{equation}
\begin{equation}\label{antico}\begin{array}{rcl}S_{\sg,\tau}:H(\sg,\tau)&\longrightarrow & H(\tau,\sg)\\
x&\longmapsto &\sg(x_1,S(x_2))S(x_3)\tau^{-1}(S(x_4),x_5).\end{array}\end{equation}
It is routine to check that the inverse of $S_{\sg,\tau}$ is given as follows:
\begin{equation}\label{anticoinverse}\begin{array}{rcl}S^{-1}_{\sg,\tau}:H(\tau,\sg)&\longrightarrow & H(\sg,\tau)\\
x&\longmapsto &\sg^{-1}(x_5,S^{-1}(x_4))S^{-1}(x_3)\tau(S^{-1}(x_2),x_1).\end{array}\end{equation}

The \it{$2$-cocycle cogroupoid} of $H$, denoted by $\underline{H}$, is the cogroupoid defined as follows:
\begin{enumerate}
\item[(i)] $\t{ob}(\underline{H})=Z^2(H)$.
\item[(ii)] For $\sg,\tau\in Z^2(H)$, the algebra $\underline{H}(\sg,\tau)$ is the algebra $H(\sg,\tau)$ defined in (\ref{mulco}).
\item[(iii)] The structural maps $\bdt^\bullet_{\bullet,\bullet}$, $\vps_\bullet$ and $S_{\bullet,\bullet}$ are defined in (\ref{comulco}), (\ref{couco}) and (\ref{antico}) respectively.
\end{enumerate}
Following \cite[Lemma 3.13]{bi1}, the morphisms $\bdt^\bullet_{\bullet,\bullet}$, $\vps_\bullet$ and $S_{\bullet,\bullet}$ indeed satisfy the conditions required for a cogroupoid. It is clear that a 2-cocycle cogroupoid is connected.

For a group $\bgm$, we denote by  ${}_\bgm^\bgm \mathcal {YD}$ the category of Yetter-Drinfeld modules over the group algebra $\kk\bgm$. If $\bgm$ is an abelian group, then it is well-known that a Yetter-Drinfeld module over the algebra $\kk\bgm$ is just a $\bgm$-graded $\bgm$-module.


 We fix the following terminologies.
\begin{itemize}
\item  a free abelian group $\bgm$ of finite rank $s$;
\item  a Cartan matrix $\mathbb{A}=(a_{ij})\in \mathbb{Z}^{\tt\times \tt}$ of finite type, where $\tt\in\NN$. Let  $(d_1,\cdots, d_\tt)$ be a diagonal matrix
of positive integers such that $d_ia_{ij} = d_ja_{ji}$, which is
minimal with this property;
\item a set $\mathcal {X}$ of connected components of the Dynkin diagram corresponding
to the Cartan matrix $\mathbb{A}$. If $1\se i, j\se \tt$, then $i\sim
j$ means that they belong to the same connected component;
\item a family $(q_{_I})_{I\in \mc{X}}$ of elements in $\kk$ which are \textit{not} roots of unity;
\item elements $g_1,\cdots , g_\tt\in \bgm$ and characters $\chi_1,\cdots, \chi_\tt\in \hat{\bgm}$ such that
\begin{equation}\label{q}\chi_j(g_i)\chi_i(g_j)=q_{I}^{d_ia_{ij}}, \t{  } \chi_i(g_i)=q_{I}^{d_i}, \t{   for all $1\se i,j\se \tt$, $I\in
\mc{X}$}.\end{equation}
\end{itemize}

For simplicity, we write $q_{ji}=\chi_i(g_j)$. Then Equation (\ref{q}) reads as follows:
\begin{equation}\label{q1}q_{ii}=q_I^{d_i}\t{ and } q_{ij}q_{ji}=q_{I}^{d_ia_{ij}}\t{ for all }
1\se i,j\se \tt, I\in \mc{X}.
\end{equation}

Let $\mc{D}$ be the collection $\mc{D}(\bgm,(a_{ij})_{1\se i,j\se
\tt}, (q_{_{I}})_{I\in \mathcal {X}}, (g_i)_{1\se i\se
\tt},(\chi_i)_{1\se i\se \tt} )$. A \textit{linking datum}
$\lmd=(\lmd_{ij})$ for $\mc{D}$ is a collection of elements
$(\lmd_{ij})_{1\se i<j\se \tt,i\nsim j}\in\kk$ such that
$\lmd_{ij}=0$ if $g_ig_j=1$ or $\chi_i\chi_j\neq\varepsilon$. We
write the datum $\lmd=0$, if $\lmd_{ij}=0$ for all $1\se i<j\se
\tt$. The datum $(\mc{D},\lmd)=(\bgm,(a_{ij}) , q_{_I},
(g_i),(\chi_i), (\lmd_{ij}) )$ is called a \textit{ generic datum of
finite Cartan type} for group $\bgm$.

A generic datum of finite Cartan type for a group $\bgm$ defines a Yetter-Drinfeld module over the group algebra $\kk\bgm$. Let $V$ be a vector space with basis $\{x_1,x_2,\cdots,x_\tt\}$. We set
$$|x_i|=g_i,\;\;g(x_i)=\chi_i(g)x_i, \;\;1\se i\se \tt, g\in \bgm,$$
where $|x_i|$ denote the degree of $x_i$.
This makes $V$ a Yetter-Drinfeld module over the group algebra $\kk\bgm$. We write $V=\{x_i,g_i,\chi_i\}_{1\se i\se \tt}\in {}^\bgm_\bgm\mc{YD}$.
The braiding is given
by
$$c(x_i\ot x_j)=q_{ij}x_j\ot x_i, \;\;1\se i,j\se \tt.$$

The tensor algebra $T(V)$ on $V$ is a natural graded braided Hopf algebra in ${}_\bgm^\bgm \mathcal {YD}$. The smash product $T(V)\#\kk\bgm$ is a usual Hopf algebra. It is also called a bosonization of $T(V)$ by $\kk\bgm$.

\begin{defn} Given a generic datum of finite Cartan type $(\mc{D},\lmd)$ for a group $\bgm$. Define $U(\mc{D},\lmd)$ as the quotient Hopf algebra of the smash product $T(V)\#\kk\bgm$ modulo the ideal generated by
$$(\t{ad}_cx_i)^{1-a_{ij}}(x_j)=0,\;\;1\se i\neq j\se \tt,\;\;i\sim j,$$
$$x_ix_j-\chi_j(g_i)x_jx_i=\lmd_{ij}(g_ig_j-1),\;\;1\se
i<j\se \tt,\;\;i\nsim j,$$
where $\t{ad}_c$ is the braided adjoint representation defined  in
\cite[Sec. 1]{as3}.
\end{defn}

The algebra $U(\mc{D},\lmd)$ is a cocycle deformation of $U(\mc{D},0)$. That is $U(\mc{D},\lmd)=U(\mc{D},0)^\sg$, where $\sg$ is the cocycle defined by
\begin{equation}\label{eq 7}\begin{array}{rcl}\sg(g,g')&=&1,\\\sg(g,x_i)&=&\sg(x_i,g)=0,\;\;1\se i\se \tt, g,g'\in \bgm.\\\sg(x_i,x_j)&=&\begin{cases}\lmd_{ij},&i<j,i\nsim j\\0,&otherwise.\end{cases}\end{array}\end{equation}

To present the CY property of the algebras $U(\mc{D},\lmd)$, we recall the concept of root vectors.
Let $\Phi$ be the root system corresponding to the Cartan matrix
$\mathbb{A}$ with $\{\al_1,\cdots, \al_\tt\}$ a set of fix simple
roots, and  $\mathcal {W}$ the Weyl group. We fix a reduced
decomposition of the longest element $w_0=s_{i_1}\cdots s_{i_p}$ of
$\mathcal {W}$ in terms of the simple reflections.  Then the
positive roots are precisely the followings,
$$\bt_1=\al_{i_1}, \;\;\bt_2=s_{i_1}(\al_{i_2}),\cdots, \bt_p=s_{i_1}\cdots s_{i_{p-1}}(\al_{i_p}).$$ For
$\bt_i=\sum_{i=1}^{\tt} m_i\al_i$, we write
$$g_{\bt_i}=g_1^{m_1}\cdots g_\tt^{m_\tt} \t{ and }\chi_{\bt_i}={\chi}_1^{m_1}\cdots {\chi}_\tt^{m_\tt}.$$

Lusztig defined the root vectors for a quantum group $U_q(\mathfrak{g})$ in
 \cite{l}. Up to a nonzero scalar, each root vector can be
expressed as an iterated braided commutator. In \cite[Sec.
4.1]{as4}, the root vectors were generalized on a pointed Hopf
algebras $U(\mc{D},\lmd)$. For each positive root $\bt_i$, $1\se i
\se p$, the root vector $x_{\bt_i}$ is defined by the same iterated
braided commutator of the elements $x_1, \cdots , x_\tt$, but with
respect to the general braiding.

\begin{rem}\label{root} If $\bt_j=\al_l$, then we have $x_{\bt_j}=x_l$. That is, $x_1,\cdots,x_\tt$ are the simple root vectors.
\end{rem}

\begin{lem}\label{lem pcy}
Let $(\mc{D},\lmd)$ be a generic datum of finite Cartan type for a group $\bgm$, and $H$
the Hopf algebra $U(\mc{D},\lmd)$. Let $s$ be the rank of $\bgm$ and $p$
the number of the positive roots of the Cartan matrix.
\begin{enumerate}
\item[(i)] The algebra $H$ is Noetherian AS-regular
of global dimension $p+s$. The left
homological integral module $\int^l_H$ of $H$ is isomorphic to
$\kk_\xi$, where $\xi:H\ra \kk$ is an algebra homomorphism defined
by $\xi(g)=(\prod_{i=1}^p\chi_{_{\bt_i}})(g)$ for all $g\in \bgm$
and $\xi(x_k)=0$ for all $1\se k\se \tt$.
\item[(ii)] The algebra $H$ is twisted CY with Nakayama automorphism $\mu$ defined by  
$\mu(x_k)=q_{kk}x_k$, for all
$1\se k\se \tt$, and $\mu(g)=(\prod_{i=1}^p\chi_{_{\bt_i}})(g)$ for
all $g\in \bgm$.

\end{enumerate}
\end{lem}
This Lemma shows that both $U(\mc{D},0)$ and  its cocycle deformation $U(\mc{D},\lmd)$ are twisted CY. Actually, Lemma 2.1 in \cite{yz} shows that the algebras $U(\mc{D},\lmd)$ are Noetherian with finite global dimension. Therefore, Theorem \ref{thm main} explains why for this class of Hopf algebras, cocycle deformation preserves the  CY property.  With Lemma \ref{lem pcy}, we can write inner automorphism in Theorem \ref{thm inn} explicitly. Let  $H$ be $U(\mc{D},0)$, then $U(\mc{D},\lmd)=H^\sg$, where $\sg$ is the cocycle as defined in (\ref{eq 7}). Let $\int^l_H=\kk_\xi$ and $\int^r_{H^\sg}=\kk_\eta$ be left homological integral of $H$ and $H^\sg$ respectively, where $\xi:H\ra \kk$ and $\eta: H^\sg\ra \kk$ are algebra homomorphisms. Then the following equation holds.
$$(S_{\sg,1}\circ S_{1,\sg})^2=[\eta]_{1,\sg}^r\circ ([\xi]_{1,\sg}^l)^{-1}\circ \gamma$$
where $\gamma$ is the inner automorphism defined by $\gamma(x_k)=[\prod_{i=1}^{p}g_{\bt_i}]^{-1}(x_k)[\prod_{i=1}^{p}g_{\bt_i}]$ for $1\se k\se \theta$ and $\gamma(g)=g$ for any $g\in \Gamma$.

\subsection*{Acknowledgement} The first and the second author are  grateful to the Department of Mathematics of Zhejiang Normal University for the hospitality they received during a visit in summer 2016. The first author is supported by AMS-Simons travel grant. The second author is supported by grants from NSFC (No. 11301126, No. 11571316, No. 11671351).

\vspace{5mm}

\bibliography{}

\begin{thebibliography}{99}
\bibitem{as4} N. Andruskiewitsch and H.-J. Schneider, \it{Finite quantum groups over abelian groups of prime exponent}, Ann. Sci. Ec. Norm. Super. \textbf{35} (2002), 1--26.
\bibitem{as} N. Andruskiewitsch and H.-J. Schneider, \it {Pointed Hopf algebras},   New Directions in Hopf Algebras, MSRI Publications {\bf 43}, 1-68, Cambridge Univ. Press, 2002.
\bibitem{ars} M. Artin and W. Schelter, \it{Graded algebras of global dimension 3}, Adv. Math. {\bf 66} (2), (1987), 171--216.
\bibitem{as3} N. Andruskiewitsch and H.-J. Schneider, \it{A characterization of quantum groups}, J. Reine Angew. Math.  \textbf{577} (2004), 81--104.
\bibitem{bt} R.~Berger and R.~Taillefer, \it{Poinc\'are-Birkhoff-Witt Deformations of Calabi-Yau Algebras}, J. Noncomm. Geom. {\bf 1} (2007), 241--270.
\bibitem{bi3} J. Bichon, \it{The representation category of the quantum group of a non-degenerate bilinear form},  Comm. Algebra {\bf 31} (2003), no. 10, 4831--4851.
\bibitem{bi} J. Bichon, \it{Hochschild homology of Hopf algebras and free Yetter-Drinfeld resolutions of the counit}, Compos. Math.  {\bf 149}  (2013),  no. 4, 658--678.
\bibitem{bi1} J. Bichon, \it{Hopf-Galois objects and cogroupoids}, Rev. Un. Mat. Argentina, {\bf 55} (2014), no. 2, 11--69.
\bibitem{bi2} J. Bichon, \it{Gerstenhaber-Schack and Hochschild cohomologies of Hopf algebras}, arXiv: 1411.1942v2.
\bibitem{be} R. Bieri and B. Eckmann, Groups with homological duality generalizing Poincar\'{e} duality, Invent. Math. {\bf 20} (1973), 103--124.
\bibitem{brg} K.A. Brown and K.R. Goodearl,\it{Homological aspects of Noetherian PI Hopf algebras and irreducible modules of maximal dimension}, J. Algebra \textbf{198} (1997),  240--265.
\bibitem{bz} K. A. Brown and J. J. Zhang, \it{Dualising complexes and twisted Hochschild (co)homology for Noetherian Hopf algebras},   J. Algebra \textbf{320} (2008), no. 5, 1814--1850.
\bibitem{ce} H. Cartan and S. Eilenberg, {\it Homological Algebra}, Princeton University Press, 1999.
\bibitem{cww} A. Chirvasitu, C. Walton and X. Wang, \it{On quantum groups associated to preregular forms}, preprint, arXiv:1605.06428.
\bibitem{do} Y. Doi, \it{Braided bialgebras and quadratic algebras}, Comm. Algebra {\bf 21} (1993), no. 5, 1731--1785.
\bibitem{dvl} M. Dubois-Violette, G. Launer, \it{The quantum group of a non-degenerate bilinear form}, Phys. Lett. B {\bf 245} (1990), 175--177.
\bibitem{erz} S. Eilenberg, A. Rosenberg, and D. Zelinsky, \it{On the dimension of modules and algebras, VIII}, Nagoya Math. J. \textbf{12} (1957), 71--93.
\bibitem{g2} V. Ginzburg, \it{Calabi-Yau algebras}, arXiv:AG/0612139.
\bibitem{gz} K. R. Goodearl and J. J. Zhang, \it{Homological properties of quantized coordinate rings of semisimple groups},  Proc. London Math. Soc. {\bf 94} (2007) no. 3, 647--671.
\bibitem{ll} M. Lorenz and M. Lorenz, \it{On crossed products of Hopf algebras}, Proc. Amer. Math. Soc. \textbf{123 }(1995), no. 1, 33--38.
\bibitem{lwz} D. M. Lu, Q. S. Wu and J. J. Zhang, \it{Homological integral of Hopf algebras}, Trans. Amer. Math. Soc. \textbf{359} (2007), 4945--4975.
\bibitem{l} G. Lusztig, \it{Introduction to quantum groups}, Birkh\"{a}user, 1993.
\bibitem{man} Y. I. Manin, \it{Quantum groups and noncommutative geometry}, Universit\'e de Montr\'eal Centre de Recherches Math\'ematiques, Montreal, QC, 1988.
\bibitem{ma} A. Masuoka, \it{Abelian and non-abelian second cohomologies of quantized enveloping algebras}, J. Algebra {\bf 320} (2008), no. 1, 1--47.
\bibitem{mro} C. Mrozinski. \it{Quantum groups of $GL(2)$ representation type}, J. Noncommut. Geom. {\bf 8} (2014), no.1,  107--140.
\bibitem{ra} D. E. Radford, \it{The order of the antipode of a finite dimensional Hopf algebra is finite}, Amer. J. Math. {\bf 98} (1976), 333--355.
\bibitem{rv1} T. Raedschelders and M. Van den Bergh, \it{The Manin Hopf algebra of a Koszul Artin-Schelter regular algebra is quasi-hereditary}, preprint, arXiv:1509.03157.
\bibitem{rv2} T. Raedschelders and M. Van den Bergh, \it{The representation theory of noncommutative $\mathcal{O}(GL_2)$}, preprint, arXiv:1509.03869.
\bibitem{rrz} M. Reyes, D. Rogalski and J.J. Zhang, \it{Twisted Calabi-Yau algebras and Homological identities}, Adv. Math. {\bf 264} (2014), 308--354.
\bibitem{sch} P. Schauenburg, \it{Hopf bigalois extensions}, Comm. Algebra {\bf 24} (1996), no. 12, 3797--3825.
\bibitem{ss} S. Shnider and S. Sternberg, \it{Quantum groups. From coalgebras to Drinfeld algebras}, Graduate Texts in Mathematical Physics, II, International Press, Cambridge, MA, 1993.
\bibitem{sw} M.E. Sweedler, Integrals for Hopf Algebras, {\it Ann. of Math.} {\bf 89} (1969), 323--335.
\bibitem{vdb} M. Van den Bergh, \it{Existence theorems for dualizing complexes over non-commutative graded and filtered rings}, J. Algebra {\bf 195} (1997), 662--679.
\bibitem{we} C. Weibel, \it{An introduction to homological algebra}, Cambridge Studies in Advanced Mathematics {\bf 38}, Cambridge University Press, Cambridge, 1994.
\bibitem{ww} C. Walton, X. Wang, \it{On quantum groups associated to non-Noetherian regular algebras of dimension 2}, Math. Z.  \textbf{284}  (2016),  no. 1-2, 543--574.
\bibitem{yu} X. L. Yu, \it{Hopf--Galois objects of Calabi--Yau Hopf algebras}, J. Algebra Appl. {\bf 15}  (2016),  no. 9, 165--194.
\bibitem{yvz} X. L. Yu, F. Van Oysteayen and Y. H. Zhang \it{ Cleft extensions of Koszul twisted Calabi-Yau  algebras}, Israel J. Math.  \textbf{214}  (2016),  no. 2, 785--829.
\bibitem{yz} X. L. Yu and Y. H. Zhang, \it{The Calabi-Yau pointed Hopf algebra of finite Cartan type}, J. Noncommut. Geom. \textbf{7} (2013), 1105--1144.
\bibitem{zh} J. J. Zhang, \it{Non-Noetherian regular rings of dimension 2}, Proc. Amer. Math. Soc. {\bf 126} (1998), no. 6, 1645--1653.
\end{thebibliography}

\end{document}